\documentclass[12pt]{amsart}
\usepackage{xspace,amssymb,amsfonts,euscript,eufrak,mathrsfs}

\usepackage{srcltx,tikz,etex}
\usepackage{amsthm,amsmath}
\usepackage{nicefrac,bbm}
\usepackage{graphicx}
 \usepackage{epstopdf,ifpdf}
\ifpdf
  \usepackage{pdfsync}
\fi
\usepackage{palatino}

\title[A geometric construction of colored HOMFLYPT homology]{A geometric construction\\ of colored HOMFLYPT homology}

\usepackage[all]{xy}

  \newcommand{\nc}{\newcommand}
  \newcommand{\renc}{\renewcommand} 

\usepackage[latin1]{inputenc}

\def\to{\rightarrow}

\nc{\Br}{\mathcal{B}}
\nc{\id}{id}
\renc{\P}{\mathbb{P}}
\nc{\cO}{\mathcal{O}}
\nc{\N}{\mathbb{N}}
\renc{\H}{\mathcal{H}}
\nc{\CC}{\mathcal{C}}

\nc{\F}{\mathcal{F}}
\nc{\G}{\mathcal{G}}
\nc{\Fq}{\mathbb{F}_q}
\nc{\Fqn}{\mathbb{F}_{q^n}}
\nc{\Q}{\mathbb{Q}}

\nc{\AC}{\mathcal{A}}
\nc{\BC}{\mathcal{B}}

\nc{\uk}{\underline{\fk}}
\nc{\triright}{\stackrel{[1]}{\to}}

\nc{\Ga}{\mathbb{G}_a} 
\nc{\Gm}{\mathbb{G}_m} 

\nc{\Loc}{\mathcal{L}}
\nc{\gG}{\Gamma}

\nc{\Perv}{\mathsf{Perv}}
\nc{\kos}[2]{\EuScript{K}_{#1,#2}}
\nc{\bet}{b}
\nc{\phc}{\Phi}
\nc{\vp}{\varphi}
\nc{\betT}{b_T}
\nc{\de}{\delta}
\nc{\QT}{Q}
\nc{\HT}{S}
\nc{\KM}{\EuScript{A}}
\nc{\ep}{\epsilon}
\nc{\Bi}{\mathbf{i}}
\nc{\BB}{B\times B}
\nc{\C}{\mathbb{C}}
\nc{\R}{\mathbb{R}}
\nc{\Cs}[1]{\underline{\C}_{#1}}
\nc{\IH}{I\!H}
\nc{\IC}{\mathbf{IC}}
\nc{\Gw}{G_w}
\nc{\up}{\vp^G_H}
\nc{\Kw}{K_w}
\nc{\B}{\mathcal{B}}
\nc{\SU}[1]{\mathrm{SU}(#1)}
\nc{\SL}[1]{\operatorname{SL}(#1)}
\nc{\GL}[1]{\mathrm{GL}(#1)}
\nc{\HU}[1]{\mathbb{H}\mathrm{U}(#1)}
\nc{\rank}{\mathrm{rk}}
\renc{\t}{\mathfrak t}
\nc{\td}{\t^*}
\nc{\pur}[1]{\cF_\be^{#1}}
\nc{\Z}{\mathbb{Z}}
\nc{\g}{\mathfrak g}
\nc{\fF}{\Phi}
\nc{\HG}{H_G}
\nc{\fk}{\mathbbm{k}}
\nc{\Si}{\S i}
\nc{\hc}{\mathbb{H}^*}
\nc{\mc}{\mathcal}
\nc{\Hom}{\mathrm{Hom}}
\nc{\ti}{\tilde}
\nc{\hcBB}{\hc_{B\times B}}
\nc{\si}{\sigma}
\nc{\al}{\alpha}
\nc{\HH}{H\!H}
\nc{\Tor}{\mathrm{Tor}}
\nc{\KR}{\mc{KR}}
\nc{\SHom}{\mathscr{H}\!om}
\nc{\Supp}{\mathrm{Supp}}
\nc{\ASu}{\mathrm{Supp}'}
\nc{\tri}{\tau}
\nc{\ext}{\mathrm{ext}}
\nc{\baet}[1]{\bar{e}(#1,t)}
\nc{\bht}[1]{\bar{h}(#1,t)}
\nc{\A}{\mathcal{A}}
\nc{\AH}{\A_H}
\nc{\AG}{\A_G}
\nc{\Lotimes}{\stackrel{L}{\otimes}}
\nc{\be}{\beta}
\nc{\beBn}{\Bn\be}
\nc{\bepBn}{\Bn\be\be'}
\nc{\ga}{\gamma}
\nc{\BS}[1]{G_{#1}}
\nc{\BSi}{\BS\Bi}
\nc{\Bf}{\mathbf{f}}
\nc{\excise}[1]{}

\nc{\Ba}{\mathbf{a}}
\nc{\Bb}{\mathbf{b}}
\nc{\Bc}{\mathbf{c}}
\nc{\Bn}{\mathbf{n}}

\nc{\ubl}[1]{{P}_{#1}}
\nc{\ubr}[1]{{_{#1}}P}
\newcommand{\becircled}{\mathaccent "7017}
\nc{\oX}{\becircled X}
\nc{\cF}{\mathcal{F}}
\nc{\cG}{\mathcal{G}}
\nc{\om}{\omega}
\nc{\ublr}{\ub{\Bn}\times\ub{\Bn}}
\nc{\Sl}{S_\Bn}
\newcommand{\qbinom}[2]{\genfrac{[}{]}{0pt}{}{#1}{#2}}

\nc{\PP}{\mathbf{P}}

\nc{\D}{\mathbb{D}}

\nc{\End}{\mathrm{End}}
\nc{\codim}{\mathrm{codim}}
\nc{\Tr}{\mathrm{Tr}}
\nc{\dgmod}{\mathrm{dgMod}}
\nc{\res}{\mathrm{res}}
\nc{\ind}{\mathrm{ind}}
\nc{\ub}[1]{P(#1)}
\nc{\bsi}{\boldsymbol{\be}}
\nc{\Dbm}{D^b_{mix}}
\nc{\cE}{\mathcal{E}}
\nc{\FF}{\mathbb{F}}
\nc{\Hec}{\mathbf{H}}
\nc{\bFb}{\mathbf{F}^\bullet}
\nc{\bGb}{\mathbf{G}^\bullet}
\nc{\bM}{\mathbf{M}}
\nc{\Fr}{\mathrm{Fr}}
\nc{\bA}{\mathbb{A}}
\nc{\clo}{\hat}
\nc{\la}{\lambda}

\DeclareMathOperator{\For}{For} 

\DeclareMathOperator{\Gr}{Gr}
\DeclareMathOperator{\gr}{gr}
\nc{\Eul}{\EuScript{E}}
\nc{\K}{\EuScript{K}}
\nc{\modu}{\mathsf{mod}}
\DeclareMathOperator{\Spec}{Spec}

\nc{\st}{\mathrm{st}}

\nc{\pt}{\mathrm{pt}}
\nc{\Stosic}{Sto\v{s}i\'c\xspace}

\UseComputerModernTips 

 \makeatletter 
 \def\revddots{\mathinner{\mkern1mu\raise\p@ 
\vbox{\kern7\p@\hbox{.}}\mkern2mu 
\raise4\p@\hbox{.}\mkern2mu\raise7\p@\hbox{.}\mkern1mu}} 
\makeatother 
  \newtheorem{thm}{Theorem}[section]
  \newtheorem{defi}[thm]{Definition}
  \newtheorem{lemma}[thm]{Lemma}
  
  \newtheorem{prop}[thm]{Proposition}
  \newtheorem{cor}[thm]{Corollary}
  \newtheorem*{theorem*}{Theorem}
\nc{\helv}{\fontfamily{phv}\selectfont}
 
  \newtheorem{ex}[thm]{Example}
  \theoremstyle{remark}
  \newtheorem{remark}{Remark}

\begin{document}
\pagestyle{plain}
\begin{center}
{\LARGE\bf  A geometric construction of\\ colored HOMFLYPT homology}
\vspace{10mm}

  \begin{tabular}{c@{\hspace{20mm}}c}
    {\sc\large Ben Webster}& {\sc\large Geordie Williamson}\\
   \it Department of Mathematics,&\it Mathematical Institute,\\ 
    \it University of Oregon &\it University of Oxford
 \end{tabular}
\vspace{3mm}

Email: {\helv bwebster@uoregon.edu}\\
{\helv geordie.williamson@maths.ox.ac.uk}
\vspace{1mm}

\end{center}
\medskip

{\small
\begin{quote}
  {\it Abstract.}  The aim of this paper is two-fold.  First, we give
  a fully geometric description of the HOMFLYPT homology of Khovanov-Rozansky.
  Our method is to construct this invariant in terms of the cohomology
  of various sheaves on certain algebraic groups, in the same spirit
  as the authors' previous work on Soergel bimodules. All the
  differentials and gradings which appear in the construction of
  HOMFLYPT homology are given a geometric interpretation.

  In fact, with only minor modifications, we can extend this
  construction to give a categorification of the colored HOMFLYPT
  polynomial, {\em colored HOMFLYPT homology}.  We show that it is in
  fact a knot invariant categorifying the colored HOMFLYPT polynomial
  and that this coincides with the categorification proposed by
  Mackaay, \Stosic and Vaz.
\end{quote}
}

\section{Introduction}
\label{sec:introduction}

The {\em colored HOMFLYPT polynomial} is an invariant of links
together with a labeling or ``coloring'' of each component with a
positive integer; in particular, for knots, there is an invariant for
each positive integer. Its most important properties are
\begin{itemize}\item 
  it reduces to the usual HOMFLYPT polynomial when
  all labels are 1, and
\item colored HOMFLYPT encapsulates all Reshetikhin-Turaev invariants
  for the link labeled with wedge powers of the standard
  representation of $\mathfrak{sl}_n$, just as the HOMFLYPT polynomial
  does for the standard representation alone.
\end{itemize}

In this paper we give a geometric construction of a categorification
of this invariant, {\em colored HOMFLYPT homology.}  Like the HOMFLYPT
homology of Khovanov and Rozansky \cite{KR05}, this
associates a triply graded vector space to each colored link such that
the bigraded Euler characteristic is the colored HOMFLYPT polynomial.
In fact, we produce an infinite sequence of such invariants, one for
each page of a spectral sequence, but only the first and second pages
are connected via an Euler characteristic to a known classical invariant.

Our construction and proofs of invariance and categorification are
algebro-geometric in nature and in fact, as a special case we obtain a
new and entirely geometric interpretation of Khovanov's Soergel
bimodule construction of HOMFLYPT homology \cite{Kho05}.  

We also show that this invariant has a purely combinatorial
description via the Hoch\-schild homology of bimodules analogous to that
of Khovanov.  In fact, it coincides with the link homology proposed
from an algebraic perspective by Mackaay, \Stosic and Vaz \cite{MSV}.
Thus, the main result of our paper has an entirely algebraic
statement:
\begin{thm}
  The colored HOMFLYPT homology defined in \cite{MSV} is a
  knot invariant, and its Euler characteristic is the
  colored HOMFLYPT polynomial.
\end{thm}

Our definition also has the advantage of categorifying 
essentially all algebraic objects involved in the definition of
colored HOMFLYPT homology.  Let us give a schematic diagram for the
pieces here, with actual operations given by solid arrows, and
(de)cat\-eg\-or\-if\-ic\-ations given by dashed ones:
\begin{equation*}
  \begin{tikzpicture}[ultra thick,->,shorten >=2mm,xscale=1.3,yscale=.9]    
   \usepgflibrary{shapes}
   \begin{scope}[thick]
     \node (braids) at (0,0) [ellipse,draw]
     {$\left\{\parbox{.6in}{ \centering colored braids}\right\}$};
     \node (links) at (8,0) [ellipse,draw]
     {$\left\{\parbox{.6in}{\centering colored links}\right\}$}; \node
     (MOY) at (8,4)
     [shape=rectangle,draw]{$\left\{\parbox{.6in}{\centering MOY
           graphs}\right\}$}; \node (Hecke) at (0,4) [rectangle,draw]
     {$\pi_\be\mathbf{H}_N\pi_\be$}; \node (Cqt) at (4,5)
     [rectangle,draw] {$\C(q,t)$}; \node (XG) at (8,-4)
     [rectangle,rounded corners,draw] {$D_{G_{L}}(X_L)$}; \node
     (sheaf) at (0,-4) [rectangle,rounded corners,draw]
     {$D_{P_\be\times P_\be}(\GL{N})$}; \node (gVect) at (4,-5)
     [rectangle,rounded corners,draw] {$\mathsf{g^3Vect}$};
   \end{scope}

   \draw (braids) to (links)node[pos=.45,above]{$\be\mapsto\clo\be$}; 
   \draw (braids) to (Hecke);
   \draw (braids) to (sheaf) node[midway,left]{$\fF_\be$};
   \draw (Hecke) to (Cqt.west) node[midway,above left ]{$\Tr_{JO}$};
   \draw (sheaf) to (gVect.west) node[near end,below left]{$\hc_{(P_{\be})_{\Delta}}(\GL N;-)$};
   \draw (links) to (MOY);
   \draw (links) to (XG) node[midway,right]{$\F_L$};
   \draw (MOY) to (Cqt.east) node[midway,above right]{eval};
   \draw (XG) to (gVect.east) node[near end,below right]{$\hc_{G_{L}}(X_{L};-)$};
   \draw (braids.north east) to[out=45,in=270] (Cqt.-135)
   node[pos=.4,draw,thick,fill=white,sloped,rotate=-8]{HOMFLYPT};
   \draw (braids.south east) to[out=-45,in=90] (gVect.135) node[rotate=5,sloped,midway,draw,thick,rounded corners,fill=white]{$\KM_2(\clo\be)$};
   \draw[loosely dashed] (sheaf.north west) to[out=120,in=-120] (Hecke.south west);
   \draw[loosely dashed] (XG.north east) to[out=60,in=-60] (MOY.south east);
   \draw[loosely dashed] (gVect.45) to[in=-70,out=70] (Cqt.-45);
\end{tikzpicture}
\end{equation*}
The top half of the diagram shows two different definitions of the colored HOMFLYPT polynomial:
\begin{itemize}
\item The path through \{MOY graphs\} is the description of the
  colored HOMFLYPT polynomial by \cite{MOY}: one associates to a link
  diagram a sum of weighted trivalent graphs, and then defines an
  evaluation function on such graphs, which in turn gives a state sum interpretation of the colored HOMFLYPT polynomial.
\item The path through $\pi_\be\mathbf{H}_N\pi_\be$ is described by \cite{LZ}: to
  each closable colored braid $\be$, we have an associated element of the
  Hecke algebra $\mathbf{H}_N$ where $N$ is the colored braid index of
  $\be$ (the sum of the colorings of the strands).  In fact, this element lies in a certain subalgebra $\pi_\be\mathbf{H}_N\pi_\be$ where $\pi_\be$ is a projection which depends on the coloring of $\be$.  The colored HOMFLY
  polynomial is obtained by applying a certain trace
  $\Tr_{JO}$ defined by Ocneanu \cite{Jon87} on $\mathbf{H}_N$.
\end{itemize}
In this paper, we show how to categorify both of these paths, as is
schematically indicated in the bottom half of the diagram, and briefly
summarized in Section \ref{sec:from-knot-diagrams}.
\begin{itemize}\item 
  The left-most dashed arrow is the isomorphism of
  $\pi_\be\mathbf{H}_N\pi_\be$ with the Grothendieck group of
  bi-equivariant sheaves for the left and right multiplication of a
  subgroup of block upper-triangular matrices $P_\beta$ on $\GL N$.
\item The right-most dashed arrow is a bijection between MOY graphs for a
  link diagram $L$ and a certain collection of simple perverse sheaves
  on a variety $X_L$ which is equivariant for the action of a group
  $G_L$, both depending on the link diagram and to be defined later.
  These are the composition factors of a perverse sheaf assigned to
  the link itself.
\item The central dashed arrow simply indicates taking bigraded
  Euler characteristic of a tri-graded vector space with respect to
  one of its gradings.
\end{itemize}

We must also show that this diagram, including the dashed arrows
``commutes.''  This follows from a result of the authors
giving a similar construction of a Markov trace for the Hecke algebra
of any semi-simple Lie group, shown in the paper 
\cite{WWmar}.

As should be clear from the above, the techniques we
use are those of algebraic geometry and geometric representation
theory.  While these are not familiar to the average topologist, we
have striven to make this paper accessible to the novice, at least if
they are willing accept a few deep results as black boxes.  As a
general rule, our actual calculations are simple and quite geometric
in nature; however, we must cite rather serious machinery
to show that these calculations are meaningful.


\subsection{} Let us briefly indicate the geometric setting in which
we work.  All material covered here is discussed in greater detail in Section
\ref{sec:mixed}.
 
Let $X$ be an algebraic variety defined by equations with 
integer coefficients. (In this paper, our varieties are built
from copies of the general linear group, so we can alway describe them in terms of integral equations.)
To $X$ one may associate a derived category $D^b(X)$ of sheaves with
constructible cohomology.  There are numerous technicalities in the
construction of this category, but we postpone discussion of these
until Section \ref{sec:mixed}.

The category $D^b(X)$ behaves similarly to the the bounded derived
category of constructible sheaves on the complex algebraic variety
defined by these equations. However, since we used integral equations,
we have an alternate perspective on these varieties; one can also
reduce modulo a prime $p$, and work over the finite field $\FF_p$. The
objects in $D^b(X)$ can also be interpreted as sheaves on these
varieties in characteristic $p$, and for technical reasons, this is
the perspective we will take.  In this situation, there is an
extra structure which helps us to understand our complexes of sheaves:
an action of the Frobenius $\Fr$ on our variety.

The category $D^b(X)$ contains a remarkable abelian subcategory 
$P(X)$ of ``mixed perverse sheaves''. For us the most important feature of 
of $P(X)$ is that every object of $P(X)$ has a canonical
``weight filtration'' with semi-simple subquotients, which is defined
using the Frobenius.  

As with any filtration, this leads to a spectral sequence
\[ E_1^{p,q} = \mathbb{H}^{p+q}( \gr^W_{-p} \F) \Rightarrow
\mathbb{H}^{p+q}(\F).\] 

Each term on the left hand side also carries an action of Frobenius
induced by that on the variety.  Considering the norms of the
eigenvalues of Frobenius may be used to give an additional grading to
each page of the spectral sequence.
It follows that each page of the spectral sequence is
\emph{triply} graded.

We will need to consider a generalization of this category, which is a
version of equivariant sheaves for the action of an affine algebraic
group on $X$.  While in principle, the technical difficulties of
understanding such a category could be resolved by working in the
category of stacks, we have found it less burdensome to give a careful
definition of the mixed equivariant derived category from a more
elementary perspective.  For the sake of brevity, this has been done
in a separate note \cite{WWequ}.

\subsection{}
\label{sec:from-knot-diagrams}

In order to apply the above machinery to knot theory, we must define a
sheaf associated to a link.  More precisely, as we discuss in
Section \ref{sec:descr-vari}, to any projection $L$ of a colored link, we
associate the natural graph $\Gamma$ with vertices given by crossings
and edges by arcs.  To this graph, we associate a variety $X_{L}$
together with the action of a reductive group $G_{L}$.
Remembering the crossings in $L$ allows us to construct a
$G_L$-equivariant mixed shifted perverse sheaf $\F_{L} \in D^b_{G_L}(X_L)$. We 
then show that $\F_L$ may be used to construct a series of knot
invariants.

Associated to any filtration on $\F_L$ (as a perverse sheaf), we have
a canonical spectral sequence converging to
$\hc_{G_{L}}(X_{L};\F_L)$. Furthermore, we can endow
$\hc_{G_{L}}(X_{L};-)$ of any mixed sheaf with the weight grading,
which is preserved by all spectral sequence differentials, so we can
think of any page of this spectral sequence as a triply-graded vector
space, where two gradings are given by the usual spectral sequence
structure, and the third by weight.

We call the spectral sequence associated
to the weight filtration {\bf chromatographic}.

\begin{thm}\label{invariance} 
  If $L$ is the diagram of a closed braid, then all pages $E_i$ for $i
  \ge 2$ of the spectral sequence computing $\hc_{G_{L}}(X_{L};\F_L)$
  associated to the weight filtration is an invariant of $L$, up to an
  overall shift in the grading.
\end{thm}

This description has a similar flavor to that of \cite{KR05} or
\cite{BN}: it begins by assigning a simple object to a single
crossing, and then an algebraic rule for gluing crossings together
(this process can be formalized as an object called a {\bf canopolis}
as introduced by Bar-Natan \cite{BN}; we will discuss this perspective in Section \ref{sec:canopolis}).  However, other papers, such as
\cite{Kho05} or \cite{MSV} have used a description which depended
on the link diagram chosen being a closed braid. In order to
show that our invariants coincide with those of \cite{MSV}, we must
find a geometric description of this form.

Assume that $\be$ is a closable colored braid with coloring given by
positive integers, $\clo{\be}$ its closure and let $N$ be the colored
braid index (the sum of the colorings over the strands of the braid).
Let $P_\be$ be the block upper triangular matrices inside $G_N$ with
the sizes of the blocks given by the coloring of the strands of $\be$.
Using left and right multiplication, we obtain a natural $P_\be\times
P_{\be}$ action on $G_N$.  We let $(P_{\be})_\Delta$ be the diagonal
subgroup, which acts on $G_N$ be conjugation.  By the classical theory
of characteristic classes, we have a canonical isomorphism of
$H^*(BP_\be)$ to partially symmetric polynomials corresponding to the
block sizes of $P_\be$, which we will use freely from now on.
\begin{thm}
  For each $\be$, there is a $P_\be\times P_\be$-equivariant
  complex of sheaves $\fF_\be$ on $\GL{N}$ with a natural 
  filtration, such that the associated spectral sequence computing 
  $\hc_{(P_\be)_{\Delta}}(\GL{N}; \fF_{\be})$ is
  canonically isomorphic 
  to the spectral sequence 
  obtained from the weight filtration for
   $\hc_{G_{\clo{\be}}}(X_{\clo{\be}};\F_{\clo{\be}})$.

Furthermore, we have an isomorphism of the $E_1$ page of the spectral
sequence for the hypercohomology $\hc_{P_\be\times P_\be}(\GL{N}; \fF_\be)$ as a complex
of bimodules over $H^*(BP_\be)$ to the complex of singular Soergel
bimodules considered by Mackaay, \Stosic and Vaz in \cite[\S 8]{MSV}.
\end{thm}

Singular Soergel bimodules have been defined and classified in the
thesis of the second author \cite{Wil}
and in the context of Harish-Chandra bimodules in \cite{StGolod}.
Since previous work of the
authors \cite{WW} has related Hochschild homology to conjugation
equivariant cohomology, we can identify our geometric knot invariant
in terms of such bimodules.

\begin{thm}\label{comparison}
  If $L$ is a closed braid, then the $E_2$-page of our spectral
  sequence is the categorification of the colored HOMFLYPT polynomial
  proposed in \cite{MSV}.

  If all the labels on the components of $L$ are 1, then this agrees
  with the triply-graded link homology as defined by Khovanov and
  Rozansky in \cite{KR05}.
\end{thm}

The higher pages of this spectral sequence are not easy to compute,
and it is not known what their Euler characteristics are.  Whether
they correspond to any classical link invariant is unknown.

\subsection*{Acknowledgments}
We would like to thank: Wolfgang Soergel for his observation that
``{Komplexe von Bimoduln sind die Gewichtsfiltrierung des armen
  Mannes}'' (``{\em Complexes of bimodules are the poor man's weight
filtration}'')
which formed a starting point for this work; Marco
Mackaay for suggesting that it could be generalized to the colored
case and explaining the constructions of \cite{MSV}; Rapha\"el Rouquier and Olaf Schn\"urer for illuminating
discussions; and Catharina Stroppel, Noah Snyder and Carl Mautner for comments on an
earlier version of this paper.
Part of this research was conducted whilst G.W.\ took
part in the program ``Algebraic Lie Theory'' at the Isaac Newton
Institute, Cambridge. B.W.\ was supported by an NSF Postdoctoral Fellowship.

\section{Description of the varieties}
\label{sec:descr-vari}

We start by recalling the steps involved in our categorification,
beginning with a braidlike diagram $L$ of an oriented colored link:
\begin{itemize}
\item To $L$ we associate a reductive group $G_L$ together with a
  $G_L$-variety $X_L$, which only depends on the graph $\Gamma$
  obtained from the diagram $L$ by forgetting under- and
  overcrossings.
\item The crossing data allows us to define a
  $G_L$-equivariant sheaf $\F_L$ on $X_L$.  
\item This sheaf $\F_L$ has a
  chromatographic spectral sequence converging to the
  $G_L$-equivariant hypercohomology of $\F_L$.
\item Each page $E_i$ of this spectral sequence for $i \ge 2$  is a
  knot-invariant (up to overall shift) and the
  $E_2$ page categorifies the colored HOMFLYPT polynomial.
\end{itemize}
In this section we discuss the first step.

\subsection{} First let us fix some notation. We fix throughout a
chain of vector spaces $0 \subset V_1 \subset V_2 \subset V_3 \subset
\cdots$ over $\FF_q$ such that $\dim V_i =i$ for all $i$. Let
\begin{equation*}
G_{i_1,\dots,i_n} := \GL{i_1}\times \cdots \times \GL{i_n},
\end{equation*}
and let
$P_{i_1,\dots,i_n }$ be the block upper-triangular matrices with
blocks $\{i_1,\dots,i_n\}$. We may identify $P_{i_1, i_2, \dots, i_n}$
with the stabilizer in $G_{i_1 + \dots + i_n}$ of the standard partial flag
\begin{equation*}
\{0 \subset V_{i_1}
\subset V_{i_1 + i_2} \subset \dots \subset V_{i_1 + \dots + i_n}\}.
\end{equation*}

Let $L$ be a diagram of an oriented tangle with marked points, with no
marked points occuring at a crossing. Let $\Gamma$ be the oriented
graph obtained by the diagram's projection, with vertices corresponding to 
crossings and marked points in $L$. That is, we simply forget the over and 
undercrossings in $L$.  We deal with
the exterior ends of the tangle in a somewhat unconventional manner;
we do not think of them as vertices in the graph, so we think of the
arcs connecting to the edge as connecting to 1 or 0 vertices. By
adding marked points to $L$ if necessary, we may assume that every
component of $\Gamma$ contains at least one vertex.

Recall that to the diagram $\Gamma$ we wish to associate a variety
$X_L$ acted on by an algebraic group $G_L$. Let us write
$\EuScript{E}(\gG)$ and $\EuScript{V}(\gG)$ for the edges and vertices
of $\gG$ respectively. Given an edge $e \in \EuScript{E}(\gG)$ write
$G_e$ for $G_i$, where $i$ is the label on $e$. Similarly, given $v
\in \EuScript{V}(\gG)$ write $G_v$ for $G_i$ where $i$ is the sum of
the labels on the incoming vertices at $v$. We define
\[
X_L  := \prod_{v \in  \EuScript{V}(\gG)}
 G_v 
\qquad
\text{ and }\qquad
G_L  := \prod 
_{e \in  \EuScript{E}(\gG)}
G_e. \]

It remains to describe how $G_L$ acts on $X_L$. Locally, near any crossing, $\Gamma$ is isotopic to a graph of the form:
\[
\xymatrix@R=0.7cm@C=0.8cm{
\ar[dr]_(.3){e_1} &  &  \\
 & v  \ar[dr]^(.7){e_4} \ar[ur]_(.7){e_2} &  \\
\ar[ur]^(.3){e_3} & & }
\]
We will call $e_1$ and $e_2$ {\bf upper} and $e_3$ and $e_4$ {\bf lower} edges
with respect to the vertex $v$. Whenever a vertex $v$ lies on an edge
$e$ we define an inclusion map $i_e : G_e \to G_v$ which is the
identity if $v$ corresponds to a marked point, and is the composition
of the canonical inclusions
\begin{gather*} G_i \hookrightarrow G_{i,j} \hookrightarrow G_{i+j} \quad \text{if $e$ is upper,} \\
G_i \hookrightarrow G_{j,i} \hookrightarrow G_{i+j} \quad \text{if $e$ is lower.}
\end{gather*}
That is, $G_e$ is included as the upper left or lower right block matrices in $G_v$, according to whether $e$ is upper or lower.

We now describe how $G_L$ acts on $X_L$ by describing the action componentwise. Let $g \in G_e$ and $x \in G_v$. We have
\[
g \cdot x = \begin{cases}
x & \text{if $v$ does not lie on $e$,} \\
x i_e(g)^{-1} & \text{if $e$ is outgoing at $v$,} \\
i_e(g) x & \text{if $e$ is incoming at $v$.} \end{cases} \]

\begin{ex} Here are two examples of $X_L$ and $G_L$:
\begin{itemize}
\item If $L$ is the standard diagram of the unknot labeled $i$ with
  one marked point 

  \begin{equation*}
 \xy 0;/r.8pc/:
(0,0)*{\bullet}; (0,0); **\crv{(1.5,4)&(8,0)&(1.5,-4)}; ?(.5)*\dir{>}; ?(.3)+(.8,.8)*{\scriptstyle i}
\endxy
\end{equation*}
we have $X_L = G_L = G_i$ and $G_L$ acts on $X_L$ by conjugation.
\item Let $L$ be the a diagram of an $(i,j)$-crossing:
\[ 
\begin{array}{c} \xymatrix@=0.4cm{
        \ar@{-}[dr]_<i & &  \\
        & \ar[dr] \\
        \ar[uurr]^<j & & &  } \end{array}
\]
Here $X_L = G_{i+j}$ and $G_D = G_i \times G_j \times G_j \times G_i$ and $(a,b,c,d)$ acts on $x \in G_{i+j}$ by
\[
\left ( \begin{array}{cc}a & 0 \\ 0 &  b \end{array} \right )
x
\left ( \begin{array}{cc}c^{-1} & 0 \\ 0 & d^{-1} \end{array} \right ).
\]
\end{itemize}
\end{ex}

\excise{
\begin{defi}
The variety $X_\gG$ is the product over vertices $$\prod_{v\in \EuScript{V}(\gG)}GL(V_v).$$
\end{defi}
\begin{defi}
  The group $G_\gG$ is the product over edges 
    $$
    \prod_{e\in  \EuScript{E}(\gG)}GL(V_e).
    $$ 
    This group has a product decomposition $G_\gG=G_\gG^\iota\times
    G_\gG^\varepsilon$, for the factors corresponding to the interior
    and exterior vertices.
\end{defi}
We can define an action of $G_\gG$ on $X_\gG$ by defining it on each factor;
the action on  $GL(V_v)$ is as follows:
\begin{itemize}
\item If the edge $e$ is not adjacent to $v$, then the action is trivial. 
\item If $v=\om(e)$, then $\GL{V_e}$ acts on $\GL{V_v}\cong
  \GL{V_e\oplus V_{e'}}$ by inverse right multiplication under the
  obvious inclusion $\iota:\GL{V_e}\hookrightarrow \GL{V_v}$.
\item If $v=\al(e)$, then $\GL{V_e}$ acts by left multiplication under the ``opposite'' inclusion $w_0\iota w_0:\GL{V_e}\hookrightarrow \GL{V_v}$.
\end{itemize}}

This is the variety and group that we shall use in our construction.
But before defining our invariant, we must first cover some
generalities on categories of sheaves on these varieties.

\section{Mixed and equivariant sheaves}
\label{sec:mixed}

In the rest of this paper, we will be using the
machinery of mixed equivariant sheaves.  In this section we intend to
summarize the the essential features of the theory
that are necessary for us, and to indicate to the reader where the details can be found.

\subsection{} 

An important point underlying what follows is that cohomology
of a complex algebraic variety (as well as most variations, such as
equivariant cohomology, or intersection cohomology) has an additional natural
grading, the {\bf weight grading}.  This grading is difficult to describe explicitly without using methods over characteristic $p$ (as we will later), but is best
understood by two simple properties:
\begin{itemize}
\item The weight grading is
  preserved by cup products, pullback and all maps in
  long exact sequences (in fact, by all differentials in any Serre
  spectral sequence).
\item This weight grading is equal to the cohomological grading on smooth projective varieties.
\end{itemize}
\begin{ex}[The cohomology of $\C^*$]
  If we write $\C\P^1$ as the union of $\C$ and $\C\P^1-\{0\}$, then in
  the Mayer-Vietoris sequence, we have an isomorphism $H^2(\C\P^1)\cong
  H^1(\C^*)$.  Thus, the cohomological and weight gradings do not agree on $H^1(\C^*)$.
\end{ex}

We plan to describe homological knot invariants using
the equivariant cohomology of varieties and the weight grading will be
necessary to give all the gradings we expect on our knot homology.

\subsection{Sheaves and perverse sheaves} We must use a
generalization of the weight grading, the weight filtration on a mixed perverse
sheaf. References for this section include \cite{SGA4}, \cite{SGA4h}, \cite{BBD} and \cite{KW}. Although we will not use it below, we should also point out that there is a way to understand mixed perverse sheaves which only uses characteristic 0 methods (Saito's mixed Hodge modules
\cite{Saito86}; see the book of Peter and Steenbrink \cite{PeSt}).

Let $q=p^e$ be a prime power.  We consider throughout a finite field $\FF_q$ with
$q$ elements and an algebraic closure $\FF$ of $\FF_q$. Unless we state otherwise
all varieties and morphisms will be be defined
over $\FF_q$. Given a variety $X$ we will write
$X \otimes \FF$ for its extension of scalars to $\FF$.

We fix a prime number $\ell \ne p$ and let $\fk$ denote the algebraic
closure $\overline{\Q_{\ell}}$ of the field of $\ell$-adic numbers.
Throughout we fix a square root of $q$ in $\fk$ and denote it by
$q^{\nicefrac{1}{2}}$.  Given a variety $Y$ defined over $\FF_q$ or
$\FF$ we denote by $D^b(Y)$ (resp. $D^+(Y)$) the bounded
(resp. bounded below) derived category of constructible $\fk$-sheaves
on $Y$ (see \cite{SGA4h}). By abuse of language we also refer to
objects in $D^b(X)$ or $D^+(X)$ as sheaves. Given a sheaf $\F$ on $X$
we denote by $\F \otimes \FF$ its extension of scalars to a sheaf on
$X \otimes \FF$. Given a sheaf $\F$ on $X$ we abuse notation and write
\[ \hc(\F) := \hc(X \otimes \FF, \F \otimes \FF) = \hc(\F \otimes \FF). \]
We \emph{never} consider hypercohomology before extending scalars.

On the category $D^b(X)$, we have the Verdier duality functor
$\D:D^b(X)\to D^b(X)^{op}$ and for each map $f : X \to Y$, we have
Verdier dual pushforward functors
\begin{equation*}
f_*, f_! : D^b(X) \to D^b(Y)
\end{equation*}
(often denoted $Rf_*$ and $Rf_!$) and Verdier dual pullback functors
\begin{equation*}
f^*, f^! : D^b(Y) \to D^b(X).
\end{equation*}
In $D^b(X)$ we have the full abelian
subcategory $P(X)$ of {\bf perverse sheaves} (see \cite{BBD}). We will
call a sheaf $\F$ {\bf shifted perverse} if $\F[n]$ is perverse for
some $n \in \Z$.

\excise{

\begin{defi}
  An object $\F \in D^b(X)$ is perverse, if there exists a
  decomposition $X = \sqcup X_{\lambda}$ of $X$ into finitely many
  locally closed subsets $X_{\lambda}$ such that, if $j_{\lambda} :
  X_{\lambda} \hookrightarrow X$ denotes the inclusion, we have
\begin{enumerate}
\item $\H^i(j_\lambda^* \F) = 0$ for $i > -\dim X_{\lambda}$,
\item $\H^i(j_i^! \F) = 0$ for $i < -\dim X_{\lambda}$.
\end{enumerate}
We denote the full subcategory of perverse sheaves on $X$ by $\Perv(X)$.
\end{defi}}

\subsection{The Frobenius and its action on sheaves}
\label{sec:frobenius-its-action}

Given any variety $X$ defined over $\FF_q$
 we have the Frobenius morphism
\[ \Fr_q : X \to X \]
which for affine $X \subset \bA^n$ is given by $(x_1, \dots, x_n) \mapsto (x_1^q, \dots, x_1^q)$. The fixed points of $\Fr_{q^n} := (\Fr_q)^n$ are precisely $X(\FF_{q^n})$, the points of $X$ defined over $\FF_{q^n}$.

Given any $\F \in D^b(X)$ we have an isomorphism (see Chapter 5 of \cite{BBD})
\[ F_q^* : \Fr_q^* \F \stackrel{\sim}{\to} \F. \]
and obtain an induced action of $F_{q^n}^* := (F^*_q)^n$ on the stalk of
$\F$ at any point $x \in X(\FF_{q^n})$.  By considering the
eigenvalues of the action of $F_{q^n}^*$ on the stalks of $\F$ at
all points $x \in X(\FF_{q^n})$ for all $n \ge 1$, one defines 
the subcategory of {\bf mixed sheaves} $D^b_m(X)$ as well as the full
subcategories of sheaves
of {\bf weight $\le w$} and {\bf weight $\ge w$} (for
$w \in \Z$)  which we denote $D^b_{\le w}(X)$ and $D^b_{\ge w}(X)$ 
respectively (see Chapter 5 of \cite{BBD}, \cite{Del80} or the first
chapter of \cite{KW}). An object is called {\bf pure of weight $i$} if
it lies in both $D^b_{\le i}(X)$ and $D^b_{\ge i}(X)$.

\excise{
\begin{ex}[Some weight filtrations on $\P^1$]
  If we write $\C\P^1$ as the union of $\C$ and $\C\P^1-\{0\}$, then in
  the Mayer-Vietoris sequence, we have an isomorphism $H^2(\C\P^1)\cong
  H^1(\C^*)$.  Thus, the cohomological and weight gradings do not agree on $H^1(\C^*)$.
\end{ex}}

Given any mixed sheaf $\F$ on $X$ all eigenvalues $\al\in \fk$ of $\Fr_q^*$ on $
\hc(\F)$ are algebraic integers such that all complex numbers with the
same minimal polynomial have the same complex norm, which by abuse of
notation, we denote $|\al|$. As $\F$ is assumed mixed, all such norms will be $q^{i/2}$ for some $i$. Let $\mathbb{H}^*_\al(\F)\subset \hc(\F)$ be the generalized eigenspace of $\al$, and let
\[ \mathbb{H}^{*,i}(\F):=\bigoplus_{|\al|=q^{i/2}}\hc_{\al}(\F).\]

\begin{remark} The constant sheaf on $X$ has a unique mixed structure
  for which the Frobenius acts trivially on all stalks, and its
  hypercohomology is the \'etale cohomology of $X$.  The $i$-th graded
  component of $H^*(X; \fk)$ for the weight grading is
  $H^{*;i}(X;\fk)$. So, our previous discussion was a reflection of
  some of the properties of the Frobenius action on the cohomology of
  algebraic varieties.
\end{remark}

If $X = \Spec \FF_q$ then a perverse sheaf on $X$ is the same
as a finite dimensional $\fk$-vector space together with a
continuous action of the absolute Galois group of $\FF_q$. In
particular we have the {\bf Tate sheaf} $\uk(1)$ which, under the
above equivalence, corresponds to $\fk$ with action of $F_q^*$ given by
$q^{-1}$. Recall that we have fixed a square root $q^{\nicefrac{1}{2}}$ of
$q$ in $\fk$ allowing us to define the {\bf half Tate sheaf}
$\uk(\nicefrac{1}{2})$, with $F_q^*$ acting by $q^{-\nicefrac{1}{2}}$.

Given any $X$ with structure morphism $X \stackrel{a}{\to} \Spec \FF_q$ and any sheaf $\F$ on $X$ we define
\begin{equation*}
\F(m/2) := \F \otimes a^*\uk(\nicefrac{1}{2})^{\otimes m}.
\end{equation*}
The following notation will prove useful:
\[ \F \langle d \rangle = \F[d](d/2). \]
Note that $\langle d \rangle$ preserves weight.

The most important fact about mixed sheaves for our purposes is that
every mixed perverse sheaf $\F$ on $X$ admits a unique increasing filtration $W$, called the {\bf weight filtration}, such that, for all $i$,
\begin{equation*}
\gr_i^W \F := W_i \F / W_{i-1} \F
\end{equation*}
is pure of weight $i$. 

In  fact, after  extension of  scalars to  the algebraic  closure, the
extensions in  this filtration  are the only  way that  mixed perverse
sheaves can fail to be semi-simple.
\begin{thm} \label{thm-gabber}
[Gabber; \cite{BBD} Th\'eor\`eme 5.3.8]
  If $\F$ is a pure perverse sheaf on $X$ 
 then $\F \otimes \FF$ is semi-simple.
\end{thm}



\subsection{The   function-sheaf    dictionary}   \label{sec:ff}   The
eigenvalues of Frobenius on stalks are also valuable for analyzing the
structure of a given perverse sheaf.  To any mixed perverse sheaf $\F$
(or more  generally, any  mixed sheaf) one  may associate a  family of
functions on $X(\Fqn)$ given by the supertrace of the Frobenius on the
stalks of the cohomology sheaves at those points:
\begin{align*}
[ \F ]_n : X(\FF_{q^n}) & \to \fk \\
x & \mapsto \Tr( F_{q^n}^*, \F_x) :=  \sum (-1)^j \Tr ( F^*_{q^n}, \H^j(\F_x)).
\end{align*}
\begin{prop}\label{funct-sheaf}
  These functions give an injective map from the Grothendieck group of
  the category of mixed perverse sheaves to the abelian group of
  functions on $X(\Fqn)$ for all $n$.  That is, if $\F$ and $\G$ are semi-simple
  and $[\F]_n= [\G]_n$ for all $n$ then $\F$ and $\G$ are isomorphic.

\end{prop}
\begin{proof}
  The fact that these functions give a map of Grothendieck groups is
  just that all maps in the long exact sequence must respect the
  action of the Frobenius, so the supertrace is additive under
  extensions. The proof that this map is injective may be found in
  \cite[Th\'eor\`eme 1.1.2]{Lau} (see also \cite[Theorem 12.1]{KW}).
\end{proof}

This reduces the calculation of the constituents of a weight filtration to a problem of computing $[\F]_n$ for simple perverse sheaves, followed by linear algebra. Indeed, suppose that $\F, \G \in D^b(X)$ are such that
 $[ \F]_n$ and $[\G]_n$ agree for all $n$ with $\G$ semi-simple. As $[\F]_n = \sum [ \gr^W_i \F]_n$ for all $n$ we conclude that $\gr^W_i \F$ is isomorphic to the largest direct summand of $\G$ of weight $i$.

\subsection{The chromatographic complex} 
We want to explain how to move between the weight filtration and a
complex, which we term {\bf the chromatographic complex}, composed of
its pure constituents. For background, the reader is referred to
\cite[Section 1.4]{DeHoII} and \cite[Section 3.1]{BBD}.

Let $\AC$ be an abelian category with enough injectives and let
$D^+(\AC)$ denote its bounded below derived category. We may also
consider the filtered derived category $DF^+(\AC)$ whose objects
consist of $K \in D^+(\AC)$ together with a finite increasing
filtration
\[
\dots \subset W_{i-1}K \subset W_iK \subset W_{i+1}K \subset \dots
\]
(finite means that $W_iK = 0$ for $i \ll 0$ and $W_iK = W_{i+1}K$
for $i \gg 0$).

For all $p$ we define
\[
\gr^W_p K := W^pK / W^{p-1}K.
\]
More generally, for $q \le p$, let
\[
(W^p/W^q)(K) := W^pK / W^qK.
\]

For all $p$ we have a distinguished triangle
\[
\gr^W_p K \to (W^{p+1}/W^{p-1})(K) \to \gr^W_{p+1} K \triright
\]
and in particular a ``boundary'' morphism $\gr_W^{p+1} \to \gr_W^p K[1]$. 
Shifting, we obtain a sequence
\begin{equation}\label{eq-chromo1}
\dots \to \gr^W_{p+1} K[-(p+1)] \to \gr^W_{p} K[-p] \to \gr^W_{p-1}
K[-(p-1)] \to \dots
\end{equation}

\begin{lemma} The morphisms in \eqref{eq-chromo1} define a complex. \end{lemma}

\begin{proof}
After completing the (commuting) triangle
\[
\xymatrix{
& ( W^{p+1}/W^{p-1})(K) \ar[dr] \\
\gr^W_p K \ar[ur] \ar[rr] & & ( W^{p+2}/W^{p-1})(K)
}
\]
to an octahedron one sees that the morphism
\[
\gr^W_{p+2}K \to \gr^W_{p+1}K[1] \to \gr^W_{p}K[2]
\]
may be factored as
\[
\gr^W_{p+2} \to W^{p+1}/W^{p-1}(K) \to \gr^W_{p+1}K[1] \to \gr^W_{p}K[2].
\]
However, the second two morphisms form part of a distinguished triangle, and
so their composition is zero.
\end{proof}

Given any left exact functor $T : \AC \to \BC$ between abelian
categories we can consider the hypercohomology objects $R^iT(K) \in
\BC$, obtained by applying $T$ to an injective resolution of
$K$. One has a spectral sequence 
(see \cite[Theorem 2.6]{McC} or
\cite[Section 1.4.5]{DeHoII})
\begin{equation} \label{eq:chrspecseq}
  E_1^{p,q}=R^{p+q}T(\gr^W_{-p}K)\Rightarrow R^{p+q}T(K)
\end{equation}
and a diagram chase shows that the first differential of this spectral
sequence (i.e. the differential on the $E_1$-page) is the same as the differential obtained by applying
$R^qT(-)$ to the complex \eqref{eq-chromo1}.

We now apply these considerations to $D^b(X)$, where $X$ and $D^b(X)$
are as in Section \ref{sec:frobenius-its-action}.  

By work of Beilinson \cite{Bei87}, $D^b(X)$ is equivalent to the
bounded derived category of the abelian subcategory $\Perv(X)$.  Thus,
we can construct a filtration whose successive quotients are pure of
the right degrees by representing an arbitrary object $\G$ as a complex of
perverse sheaves $\F_i$, and taking the weight filtration on each.  We call
this {\bf a weight filtration} on $\G$.  As the choice of article
emphasizes, this is {\bf not} unique; it depends on how we represent
$\G$ as a complex of perverse sheaves.

Applying the above considerations to $\F$ together with its weight
filtration we obtain:

\begin{defi} The {\bf local chromatographic complex} of a
  mixed sheaf $\F \in D^b_m(X)$ is the complex
\[
\dots \to \gr^W_{p+1} \F[-(p+1)] \to \gr^W_{p} \F[-p] \to \gr^W_{p-1}
\F[-(p-1)] \to \dots
\]
Applying $T=\hc(-)$ we obtain the {\bf global chromatographic complex},
\[
\cdots \longrightarrow \hc(\gr^W_{i+1}\F[-(i+1)]) 
\longrightarrow \hc(\gr^W_i\F[-i]) 
\longrightarrow \hc(\gr^W_{i-1}\F[-(i-1)])
\longrightarrow \cdots
\]
\label{chr-ss-def}
  The spectral sequence \eqref{eq:chrspecseq} with $T=\hc(-)$
is the {\bf chromatographic spectral sequence}.
\end{defi}

Unfortunately, this definition is not entirely an invariant of the
object $\G$, but the dependence on choice of filtration is not very
strong.

\begin{prop}
  The chromatographic complexes associated to two different weight
  filtrations on a single object $\G\in D^b(X)$ are homotopy equivalent.
\end{prop}
In particular, this shows that all pages of the chromatographic
spectral sequence after the first are independent of the choice of filtration. 
\begin{proof}
  We note that if $\G$ is quasi-isomorphic to a complex
  $\cdots\to\F_i\to\cdots$, then we obtain a natural bicomplex by
  writing the chromatographic complexes of $\F_i$ vertically, and then
  the maps induced by the original differentials horizontally.  By
  Gabber's theorem, we note that every term in this bicomplex is
  semi-simple, and the horizontal maps go between objects pure of the
  same degree, and thus split.  

  Now assume perverse sheaves $\F_i'$ form another complex isomorphic
  in the derived category to $\G$.  For simplicity, we may assume
  there is a quasi-isomorphism $\phi_i:\F_i\to\F_i'$ between these
  complexes.  This induces a map $\phi^\#$ between our bicomplexes, which is an
  isomorphism after taking horizontal cohomology, since this will give
  us the chromatographic complexes of the perverse cohomology of $\G$.

  Consider the kernel of $\phi^\#$.  This is itself a bicomplex, and
  each of its rows has trivial cohomology, and is split.  Thus, each
  row is homotopic to 0.  Furthermore, we can choose these homotopies
  so that they commute with the vertical differentials, and thus when
  applied to the total complex of the kernel, they show that this
  total complex is null-homotopic.

  We now use the result that any surjective chain map
  whose kernel is homotopic to the zero complex and is a summand of
  the chain complex {\it with the differentials forgotten} is a
  homotopy equivalence (this is a consequence of Gaussian
  elimination). Thus, the chromatographic complex from the $\F_i$'s is
  homotopy equivalent to the total complex of the image of $\phi^\#$, and the
  dual result applied to the inclusion of the image shows that the
  chromatographic complex for $\F_i'$ is also equivalent to this image. 
\end{proof}

\begin{prop}
  The global chromatographic complex is preserved (up to homotopy) by proper
  pushforward.
\end{prop}
\begin{proof}
  Proper pushforward preserves purity, and thus sends weight
  filtrations to weight filtrations. Furthermore, pushforward always preserves hypercohomology.
\end{proof}

\begin{cor} If we let $E^{*,*}_*$ be the chromatographic spectral sequence, then all differentials preserve the weight grading on hypercohomology.  Furthermore, we have 
\begin{itemize}
\item $E_1^{i,j}=\mathbb{H}^{i+j}(\gr^W_{-j}\F)$ is the global chromatographic complex.
\item $E_2$ is the cohomology of the global chromatographic complex.
\item the chromatographic spectral sequence converges to the
  hypercohomology $\mathbb{H}^{i+j}(\F)$.
\end{itemize}  
\end{cor}

\begin{remark} It seems likely that it is possible to interpret the results of this section  in terms of ``weight structures'', introduced by Bondarko \cite{Bond} and Paukzsello \cite{Pauk}. In particular, Bondarko shows the existence of a functor from a derived category equipped with a suitable weight structure, to the homotopy category of pure complexes in a very general framework. \end{remark} 

\subsection{The equivariant derived category}

We have thus far discussed the theory of perverse sheaves on schemes,
but we will require a generalization of schemes which includes
the quotient of a scheme $X$ by the action of an algebraic group $G$,
which can be understood as $G$-equivariant geometry on $X$.

This quotient can be understood as a stack, but the theory
of perverse sheaves on stacks is not straightforward, and it proved
more suitable to give a treatment of the equivariant derived category
similar to that of Bernstein and Lunts \cite{BL}, but with an eye to
working over characteristic $p$ with the action of the Frobenius (that
is ``in the mixed setting'').  We have done this in a separate
note \cite{WWequ}.

The result is the {\bf bounded below equivariant derived category} $D_G^+(X)$ and its subcategory $D^b_G(X)$ of bounded complexes 
for a variety $X$ acted on by an affine algebraic group $G$.
The resulting formalism is essentially identical to that of
Bernstein and Lunts. We now summarize the essential points.

We have a forgetful functor
\[ \For : D^+_G(X) \to D^+(X) \]
which preserves the subcategories of bounded complexes and, given any $\F \in D^+_G(X)$, the cohomology sheaves of $\For(\F)$ are locally constant along the $G$-orbits on $X$.

Given an equivariant map $f : X \to Y$ of $G$-varieties we have functors
\[ f_*, f_! : D^+_G(X) \to D^+_G(Y)  \]
and 
\[ f^*, f^! : D^+_G(Y) \to D^+_G(X) \]
for equivariant maps $f : X \to Y$ of $G$-varieties. These functors commute with the forgetful functor.

If $H \subset G$ is a closed subgroup and $X$ is a $G$-space we have 
an adjoint pair $(\res_H^G, \ind_H^G)$ of restriction and induction functors
\[ \res_H^G : D_G^+(X) \to D_H^+(X) \qquad \text{ and }\qquad
\ind_H^G : D_H^+(X) \to D_G^+(X).  \]
These functors preserve the subcategories of bounded complexes, and one has an isomorphism $\res^G_{\{ 1 \}} \cong \For$.

More generally, given a map $\phi: H \to G$, a $G$-variety $X$, an
$H$-variety $Y$ and a $\phi$-equivariant map $m : X \to Y$ we have an
adjoint pair $({}^G_H m^*, {}^G_H m_*)$ of functors
\[ {}^G_H m^* : D^+_H(Y) \to D^+_G(X) \; \qquad\text{ and } \;\qquad
{}^G_H m_* : D^+_G(X) \to D^+_H(Y). \] As a special case, we have
${}_H^G\id^*\cong \res^G_H,{}_H^G\id_* \cong \ind^G_H$.  The functor ${}^G_H m^*$
preserves the subcategory of bounded complexes, but this is not true
in general for ${}^G_H m_*$. In fact, this is the reason that we are
forced to consider complexes of sheaves which are not bounded above.

If $G = G_1 \times G_2$ and $G_1$ acts 
freely on $X$ with quotient $X/G_1$ one has the {\bf quotient equivalence}
\begin{equation} \label{quot-equiv}
 D^+_G(X) \cong D^+_{G_2}(X/G_1)
\end{equation}
which restricts to an
equivalence between the subcategories of bounded complexes.  If we let
$\phi : G_1 \times G_2 \to G_2$ denote the projection then the
quotient map $X \to X/G_1$ is $\phi$-equivariant and the above
equivalence is realized by ${}_{G_1 \times G_2}^{G_2}m^*$ and ${}_{G_1
  \times G_2}^{G_2}m_*$.

Many notions carry over immediately using the forgetful
functor $\For :D^+_G(X)\to D^+(X)$.  For example, we call an
object $\F$ in $D^+_G(X)$ {\bf perverse} if and only if
$\For \F$ is perverse.

Moreover if $X$ is defined over $\FF_q$, then we can also incorporate
the action of the Frobenius.  In particular, perverse objects in
$D^+_{G}(X)$ still have weight filtrations, which are preserved by
the restriction functor and
we can extend Proposition~\ref{funct-sheaf} to the
equivariant setting.

\section{Description of the invariant}
\label{sec:descr-invar}
Equipped with these geometric tools, we continue the construction of
our invariant.

\subsection{}
In this subsection we describe the sheaf $\F_L$ on $X_L$.

We first discuss the case of a single $(i,j)$-crossing:
\[
\xymatrix@R=0.5cm@C=0.6cm{
\ar[dr]_<i &  &  \\
 & v  \ar[dr] \ar[ur] &  \\
\ar[ur]^<j & & }
\]
As we have seen $X_L = G_{i+j}$. Consider the big Bruhat cell
\begin{equation}
 U := \{ g \in G_{i+j} \; | \; V_i \cap gV_j = 0 \}\label{u-def}
\end{equation}
and let $j : U \hookrightarrow G_{i+j}$ denote its inclusion. As $U$ is an orbit under $P_{i,j} \times P_{j,i}$ it is certainly $G_L$-invariant. We now define $\F_v  = \F_L \in D_{G_L}(X_L)$ as follows:
  \begin{align*}
    \begin{array}{c}   \xymatrix@=0.4cm{
        \ar[ddrr]_<i & &  \\
        & \ar[ur] \\
        \ar@{-}[ur]^<j & & & }
        \end{array} 
        \mapsto \quad  j_*\uk_U \langle ij \rangle  \\
    \begin{array}{c} \xymatrix@=0.4cm{
        \ar@{-}[dr]_<i & &  \\
        & \ar[dr] \\
        \ar[uurr]^<j & & & } \end{array} \mapsto \quad j_! \uk_U \langle ij \rangle
  \end{align*}
As $U$ is the complement of a divisor in $G_{i+j}$ both these sheaves are shifted perverse.
  
We now consider the case of a general diagram $L$ of an oriented
colored tangle.  After forgetting equivariance, $\F_L$ is simply the
exterior product of the above sheaves associated to each crossing. To
take care of the equivariant structure we need to proceed a little
more carefully.

Let $L$ be the diagram of an oriented colored tangle and $\Gamma$ its
underlying graph. Let $L^{\prime}$ be the diagram obtained from $L$ by
cutting each strand connecting two vertices in $\Gamma$ (so that
$L^{\prime}$ is a disjoint union of $(i,j)$-crossings). Let
$\Gamma^{\prime}$ be the graph corresponding to
$L^{\prime}$. Obviously we have $X_L = X_L^{\prime}$. Note also that
for every $e$ with two vertices in $\Gamma$, we have two edges, which we denote  $e_1$ and $e_2$ in $\Gamma^{\prime}$. We have a natural map $G_L \to
G_L^{\prime}$ which is the identity on factors corresponding to edge
strands, and is the diagonal $G_e \to G_{e_1} \times G_{e_2}$ on the
remaining factors.

We define
\[ \F_{L} := \res_{G^{\prime}}^G \Big(\underset{v\in
  \EuScript{V}(\gG^{\prime})}\boxtimes\F_v \Big) \in
D^b_{G_L}(X_{L}). \]

\excise{
Let $e,e'$ be the two edges such that $v=\om(e)=\om(e')$, with $e$ on
the left side, as shown in (\ref{crossing}).
\begin{defi}
  Let
  \begin{equation}\label{u-def}
    U =\big\{g \in GL(V_v)|V_{e'} \cap w_0g\cdot V_e =\{0\}\big\}
  \end{equation}
  and $j$ denote its inclusion.
\end{defi}
Note that $U$ is invariant under the action of $G_\gG$; in fact, it's
invariant under the right action of the parabolic $P_e$ preserving
$V_e$, and the left action of the parabolic $w_0P_{e'}w_0$ preserving
$w_0V_{e'}$.

Since $X_\gG$ is a product, we can describe an equivariant sheaf on
this space by giving one on each term, and taking exterior tensor
product.  Thus, we can associate a sheaf on $X_\gG$ to the link $L$ by
associating a sheaf on $GL(V_v)$ to each vertex, depending on the type
of crossing.

\begin{defi}
  The sheaf $\F_L$ is the exterior product
  \begin{math}
\F_L=\underset{v\in \EuScript{V}(\gG)}\boxtimes\F_v
  \end{math},
where the $G_\gG$-equivariant sheaf $\F_v$ on $\GL{V_v}$ is depends on whether $v$ corresponds to a positive or negative crossing in the following manner:
  \begin{align*}
    \begin{array}{c}   \xymatrix@=0.4cm{
        \ar[ddrr]_<i & &  \\
        & \ar[ur] \\
        \ar@{-}[ur]^<j & & & }
        \end{array} 
        \mapsto \quad  j_*\uk_U \langle ij \rangle  \\
    \begin{array}{c} \xymatrix@=0.4cm{
        \ar@{-}[dr]_<i & &  \\
        & \ar[dr] \\
        \ar[uurr]^<j & & & } \end{array} \mapsto \quad j_! \uk_U \langle ij \rangle
  \end{align*}

\end{defi}}

Of course, this sheaf depends on the link diagram used; different
diagrams correspond to sheaves on different spaces.  Instead, we will
studying the hypercohomology of these sheaves, and the corresponding chromatographic spectral sequence.

\begin{defi}
  We let $\KM_{i}(L)$ denote the $i$th page of the chromatographic
  spectral sequence (as given by Definition~\ref{chr-ss-def}) for
  $\F_L$.  This is triply graded, where by convention subquotients of $\mathbb{H}^{j-\ell;{j-k}}(\gr^W_\ell\F_L)$ lies in $\KM_i^{j;k;\ell}(L)$.
\end{defi}
\begin{remark}
  These grading conventions may seem strange, but they are an attempt
  to match those already in use in the field.  These conventions are
  almost those of \cite{MSV}, though we will not match perfectly since
  we have different grading shifts in our definition of the complex
  for a single crossing.  We hope the reader finds these choices
  defensible on grounds of geometric naturality.  This simply changes
  the shift we must apply to our invariant to assure it is a true knot
  invariant.
\end{remark}
It is these spaces for $i>1$ which we intend to show are knot invariants (up to shift).

\subsection{Braids and sheaves on groups}
\label{sec:braids-sheav-groups}

As we mentioned in Section \ref{sec:introduction}, in the special case
of a braid $\be$, there is a different perspective on this construction.

Let $\be$ be the diagram of a colored braid on $n$ strands with labels
$\Bn=(i_1, i_2, \dots, i_n)$ and underlying labeled graph
$\Gamma$. Let $N = \sum_{j=1}^n i_j$ denote the colored braid
index. We assume our braid is in generic position, so reading from
start to finish, we fix an order on the vertices $v_1, v_2, \dots,
v_p$ of $\Gamma$.  This corresponds to an expression for $\beta$ in the
standard generators of the braid group.

In the previous section we described how to associate to $\beta$ a
group $G_{\be}$ and a $G_{\be}$-variety $X_{\beta}$.  We can decompose
$G_{\beta}$ as
\[ G_{\be} = G_{\beta}^+ \times G_{\beta}^\iota \times G_{\beta}^- \]
where $G_{\beta}^+$, $G_{\beta}^\iota$ and $G_{\beta}^-$ denote the factors
of $G_{\be}$ corresponding to incoming, interior and outgoing edges of $\Gamma$ respectively.

In what follows we will describe an action of $G_{\beta}^+ \times G_{\beta}^-$ on $G_N$ and a map
\begin{gather*}
m : X_{\beta} \to G_N
\end{gather*}
equivariant with respect to the natural projection $\phi: G_{\be} \to
G_{\beta}^+ \times G_{\beta}^-$.  We will study our sheaf $\F_{\be}$ by considering its equivariant pushforward under this map. 

We start by describing an embedding $\alpha_{v} : G_v \to G_N$ corresponding to each vertex $v \in \Gamma$.
Let us fix a basis $e_1, \dots, e_N$ of $V_N$ and let $W_1, W_2,
\dots, W_n$ be vector spaces (again with fixed bases) of dimensions
$i_1, i_2, \dots ,i_n$ respectively.
\begin{defi}
  Given any permutation $w \in S_n$, we let 
  \[ h_w:W=\bigoplus_{j=1}^n W_j \stackrel{\sim}{\to} V \] be the
  isomorphism defined by mapping the basis vectors of $W_{w^{-1}(1)}$
  to the first $w^{-1}(1)$ basis vectors of $V$ in their natural
  order, the basis vectors of $W_{w^{-1}(2)}$ to the next $w^{-1}(2)$
  basis vectors etc.

For any braid $\be$, we have an
  induced permutation, and by abuse of notation, we let $h_\be$ be the
  map corresponding to this permutation.
\end{defi}
In the obvious basis, this map is a permutation matrix.  The corresponding permutation is a shortest coset representative for the Young subgroup preserving the partition of $[1,N]$ of sizes $i_1,\dots, i_n$, corresponding to the ``cabling'' of the permutation $w$.

Now choose a vertex $v$ in $\Gamma$, let $e^{\prime}$ and
$e^{\prime\prime}$ denote the two incoming edges, which are in the
strands connected to the $j'$th and $j''$th incoming vertex
respectively, so $i_{j'},i_{j''}$ are the labels on $e^{\prime}$ and
$e^{\prime\prime}$. Because we have ordered the vertices of $\Gamma$,
we may factor $\be$ into braids $\al_v\cdot\be_v\cdot\om_v$ with $\be_v$ consisting
of a simple crossing corresponding to $v$. The procedure described in the
previous paragraph yields an embedding $W_{j^{\prime}} \oplus
W_{j^{\prime\prime}}\hookrightarrow W\overset{h_{\al_v}}\longrightarrow V_N$. This induces an embedding
\[ \iota_v : G_v \hookrightarrow G_N \]

We let braids on $n$ strands act on sequences of $n$ elements on the right by the
usual association of a permutation to each braid.  We may then identify
\begin{align*}
G_{\be}^+ & \cong G_{\Bn} \\
G_{\be}^- & \cong G_{\beBn}
\end{align*}
and therefore obtain an action of $G_{\beta}^+ \times G_{\beta}^-$ on
$G_N$ by left and right multiplication. We let
$P^{+}_\be=P_\Bn,P^-_{\be}=P_{\beBn}$.  We denote by $\phi: G_{\be} \to
P_{\beta}^+ \times P_{\beta}^-$ be the composition of the natural
projection with the inclusion $G^\pm_\be\hookrightarrow P^\pm_\be$.

Consider the map
\begin{align*}
m : X_{\be} & \to G_N \\
(g_{v_1}, \dots, g_{v_p}) & \mapsto \iota_{v_1}(g_{v_1})
\iota_{v_2}(g_{v_2}) \dots \iota_{v_p}(g_{v_p})
\end{align*}
It is easy to see that this map is equivariant with respect to $\phi$.

\begin{defi}
  Let $\fF_\be= {}_{G_{\be}}^{P_{\beta}^+ \times P_{\beta}^-} m_*\F_\be$.
\end{defi}

This definition is useful, since it is compatible with braid
multiplication.  We have a diagram of equivariant maps of spaces
\begin{equation*}
  \xy 
  {\ar^(.55){\mu} (20,0)*+{G_N\times G_N}="A"; (50,0)*+{G_N}="C"};
 {\ar^(.55){\pi_2} "A"; (-10,-5)*+{G_N}};
  {\ar_(.55){\pi_1} "A"; (-10,5)*+{G_N}};
\endxy
\end{equation*}
As usual, this diagram can be used to construct the functor of sheaf convolution
  \begin{equation*}
    -\star-:D^b_{P_{\Bn}\times P_{\beBn}}(G_N)\times D^b_{P_{\beBn}\times P_{\bepBn}}(G_N)\to D^b_{P_{\Bn}\times P_{\bepBn}}(G_N)
  \end{equation*}
  \begin{equation*}
\F_1\star\F_2\cong {}_{P_{\Bn}\times P_{\beBn}\times P_{\bepBn}}^{P_{\Bn}\times P_{\bepBn}}\,\mu_*\left(\res^{P_{\Bn}\times P_{\beBn}^2\times P_{\bepBn}}_{P_{\Bn}\times P_{\beBn}\times P_{\bepBn}}(\F_1\boxtimes\F_2)\right).
\end{equation*}

\begin{thm}
  We have a canonical isomorphism
  $\fF_\be\star\fF_{\be'}\cong \fF_{\be\be'}$.
\end{thm}
We should note that here we are simply claiming that this holds for the composition of diagrams.  We will prove in Sections \ref{sec:invariant} and \ref{sec:proof-invariance:-gl} that the sheaf we associate to a braid doesn't depend on the choice of presentation.
\begin{proof}
  Immediate from the definition of $\fF$.
\end{proof}

As $G^\iota_\be$ acts freely on $X_\be$, and we may factor $m$ as
\[ X_{\be} \to X_{\be} / G^\iota_\be \to G_N. \] One may verify that
the second map is the composition of an affine bundle along which
$\F_\be$ is smooth, and a proper map.  It follows that
${}_{G_{\be}}^{P_{\beta}^+ \times P_{\beta}^-} m_*$ preserves the
weight filtration on $\F_\be$. Hence the chromatographic spectral
sequences for $\F_\be$ and $\fF_\be$ are isomorphic.

Note that if $\be$ is closable, then $\beBn=\Bn$, and $P^\pm_\be$ have
the same image in the group. Thus these subgroups are canonically
isomorphic.  Let $(P_\be)_{\Delta}\subset P_{\beta}^+ \times P_{\beta}^-$ be the diagonal and let $\clo\be$ be
the colored link diagram given by the closure of $\be$.
\begin{thm}
  We have a canonical isomorphism between
  \begin{itemize}
    \item  the chromatographic spectral sequence of $\F_{\clo\be}$ as a
    $G_{\clo\be}$-sheaf and 
    \item  the chromatographic spectral sequence of $\fF_\be$ as a
    $(P_\be)_{\Delta}$-sheaf.
  \end{itemize}
\end{thm}
\begin{proof}
  Since $P_*$ and $G_*$ are homotopy equivalent, the functor
  $\res^{P_*}_{G_*}$ is fully faithful, so we may work with their 
  restrictions instead.  We have already observed
  that the weight filtration on $\fF_\be$ and the pushforward of the weight filtration on $\F_{\be}$ agree. Thus
  the equivariant chromatographic spectral sequences of
  $\res^{G_\be}_{\phi^{-1}(H)}\F_\be$ and $\res^{G^+_\be\times
    G^-_\be}_H\fF_\be$ are canonically isomorphic for any subgroup
  $H\subset G^+_\be\times G^-_\be$.

On the other hand, we have a canonical
  identification $G_{\clo\be}\cong \phi^{-1}
  \big((G_\be)_{\Delta}\big)$, and $X_\be=X_{\clo\be}$, with $\F_{\clo\be}=\res^{G_\be}_{G_{\clo\be}}\F_\be$.  The result follows.
\end{proof}

\section{Analyzing an $(m,n)$-Crossing}

\subsection{} In this section we work out all the details for an
$(m,n)$-crossing. This will be of use in expressing the invariant in terms of bimodules.

We consider an $(m,n)$-crossing. Its underlying graph is
\begin{equation*}
\xymatrix@R=0.5cm@C=0.8cm{
 \ar[dr]_m & &  \\
 & \bullet \ar[ur]_n \ar[dr]^m \\
 \ar[ur]^n & &  }
\end{equation*}
and the variety in question is $G_{m+n}$ acted on by $P_{m,n} \times P_{n,m}$ by left and right multiplication: $(p,q) \cdot g = pgq^{-1}$ for $g \in G_{n+m}$ and $(p,q) \in P_{m,n} \times P_{n,m}$. The orbits under this
action are
\begin{equation*}
\cO_i = \{ g \in G_{m+n} \; | \; \dim V_m \cap gV_n = i \} \; \text{for } 0 \le i \le \min(n,m).
\end{equation*}
Clearly $\cO_j \subset \overline{\cO_i}$ if and only if $j > i$. For all
$0 \le i \le \min(n,m)$ we denote the inclusion of the orbit $\cO_i$ by $f_i : \cO_i \hookrightarrow G_{n+m}$.

For each orbit $\cO_i$ we have the corresponding intersection cohomology complex. It will prove natural to normalize them by requiring
\[ \IC(\overline{\cO_i}) _{|\cO_i} \cong \uk_{\cO_i}\langle nm - i^2 \rangle. \]
Under this normalization each $\IC(\overline{\cO_i})$ is pure of weight 0.

We first describe resolutions for the closures $\overline{\cO_i}
\subset G_{m+n}$.  Consider the variety
\begin{equation*}
\widetilde{\cO_i} = \{ (W, g) \in \Gr_i^m  \times G_{m+n} \; | \; W \subset V_m\cap gV_n \}.
\end{equation*}
We have an action of $P_{m,n} \times P_{n,m}$ on $\widetilde{\cO_i}$ given by $(p,q) \cdot (W, g) = (pW, pgq^{-1})$. The second projection induces an equivariant map:
\begin{equation*}
\pi_i : \widetilde{\cO_i} \to \overline{\cO_i}.
\end{equation*}

\begin{prop}\label{small} This is a small resolution of singularities. \end{prop}

\begin{proof}
The morphism $\pi_i$ is patently an isomorphism over $\cO_i$. Since
$\cO_i$ is exactly the subset of $G_{n+m}$ where the induced map
$V_n\to V/V_m$ has rank $n-i$, we have that $\cO_i$ has the same
codimension  in $G_{m+n}$ as the space of rank $n-i$ matrices in $G_n$, which is $i^2$. 
Hence, for $j < i $, $\cO_i$ is of codimension $i^2 - j^2$ in $\overline{\cO_j}$. Over any $x \in \cO_j$ the fiber is the Grassmannian $\Gr_i^j$. Thus
\begin{equation*}
2 \dim \pi_i^{-1}(x) = 2i(j-i) < (j + i)(j-i) = \codim _{\overline{\cO_i}} \cO_j. \qedhere
\end{equation*}
\end{proof}

\begin{cor}\label{IC=res}
  \begin{math}
\displaystyle{\IC(\overline{\cO_i}) \cong \pi_{i*} \uk_{\widetilde{\cO_i}}\langle nm-i^2 \rangle.}
\end{math}
\end{cor}
\begin{proof}
  Proposition \ref{small} implies that $\pi_{i*}
  \uk_{\widetilde{\cO_i}}$ is a shift and twist of $\IC(\overline{\cO_i})$, since
  pushforward by a small resolution sends the constant sheaf to a shift
  of the intersection cohomology sheaf on the target.
The restriction of $\pi_{i*} \uk_{\widetilde{\cO_i}}\langle nm-i^2 \rangle$ to 
$\cO_i$ is isomorphic to $\uk_{\cO_i} \langle nm-i^2 \rangle$, which is our choice of normalization.
\end{proof}

Given sheaves $\F, \G \in D^b_G(X)$ let us write
\[
\Hom^{\bullet}(\F, \G) := \bigoplus_{m} \Hom(\F, \G[m]).
\]
This is a graded vector space.

\begin{prop} \label{prop-spectral}
In $D^b_{P_{m,n} \times P_{n,m}}(G)$ we have an isomorphism
\[
\Hom^{\bullet}(\IC(\cO_i), \IC(\cO_{i^{\prime}})) \cong \bigoplus_{j} \Hom^{\bullet}( f_j^! \IC(\cO_i), f_j^* \IC(\cO_{i^{\prime}})).
\]
\end{prop}

\begin{proof} 
For flag varieties this is \cite[Theorem 3.4.1]{BGS}. One may reduce to this situation using the quotient equivalence.
\end{proof}
 
\subsection{} Our aim in this section is to calculate the weight
filtration on the sheaves associated to positive and negative
crossings. We set $[n]_q = 1 + q + \dots + q^{n-1}$, $[n]_q ! = [n]_q [n-1]_q \dots [1]_q$ and
\[
\qquad \qbinom{j}{i}_q = \frac{[j]_q}{[j-i]_q! [i]_q! }.
 \]

In order to understand the constituents via the
function-sheaf correspondence discussed in Section \ref{sec:ff}, we
must calculate the trace of the Frobenius on the stalks of
$\IC(\overline{\cO_i})$.  Base change combined with the
Grothendieck-Lefschetz fixed point formula yields
 
\begin{cor} \label{cor:stalks}
If $j > i$ and $x \in \cO_j(\FF_{q^a})$ we have
\begin{equation*}
\Tr(F_{q^a}^*, (\pi_{i*} \uk_{\widetilde{\cO_i}})_x)= \# \Gr^j_i(\FF_{q^a})=\qbinom{j}{i}_{q^a}.
\end{equation*}
\end{cor}

In the following proposition $W$ denotes the weight filtration:

\begin{prop}\label{pure-constituents} One has isomorphisms:
\begin{align*}
\gr^W_{ - i} j_!\uk_{\cO_0} \langle nm \rangle & \cong \IC(\overline{\cO_i})(i/2) \\
\gr^W_{ i} j_*\uk_{\cO_0} \langle nm \rangle & \cong \IC(\overline{\cO_i})(-i/2)
\end{align*}
\end{prop}

\begin{proof} Because taking weight filtrations commutes with
  forgetting equivariance, it is enough to handle the non-equivariant
  case. Note also that $\IC(\cO_i)(i/2)$ is pure of weight $-i$. Thus,
  by the remarks in Section \ref{sec:ff}, the first statement of the
  proposition follows from the equality of the functions
\[
[j_!\uk_{\cO_0} \langle nm \rangle]_{q^a} = \sum_i [\IC(\cO_i)(i/2)]_{q^a} \]
for all $a \ge 1$. Evaluating at a point $x \in \cO_j(\FF_{q^a})$ we need to verify
\[
(-1)^{nm/2} \delta_{0j} q^{-anm/2}
= \sum_{0 \le i \le j} (-1)^{nm-i^2}q^{a(i^2 - nm - i)/2} \qbinom ji_{q^a} \]
or equivalently
\[
\delta_{0j} = \sum_{0 \le i \le j}(-1)^iq^{i(i-1)/2}\qbinom ji_q
\]
which is a standard identity on $q$-binomial coefficients. The second
statement follows from the first by Verdier duality.
\end{proof}

\begin{prop}\label{one-d}
We have equalities
  \begin{equation*}
\displaystyle{\dim
  \operatorname{Ext}^1\!\left(\mathbf{IC}(\cO_i),\mathbf{IC}(\cO_{i+
    1})\right)=\dim
  \operatorname{Ext}^1\!\left(\mathbf{IC}(\cO_{i+1}),\mathbf{IC}(\cO_{i})\right)=1.}
\end{equation*}

\end{prop}
\begin{proof}
  By the Verdier self-duality of $\mathbf{IC}$ sheaves, we
  have an equality of dimensions $$\dim
  \operatorname{Ext}^1(\mathbf{IC}(\cO_i),\mathbf{IC}(\cO_{i+
    1}))=\dim
  \operatorname{Ext}^1(\mathbf{IC}(\cO_{i+1}),\mathbf{IC}(\cO_{i})),$$
so we need only give a proof for one.

Using Proposition \ref{prop-spectral}, and remembering that 
\[ 
\dim
  \operatorname{Ext}^1(\mathbf{IC}(\cO_i),\mathbf{IC}(\cO_{i+
    1}))
= \dim
  \Hom(\mathbf{IC}(\cO_i),\mathbf{IC}(\cO_{i+
    1})[1])
\]
one may identify the above space with $H^{2i}(\pi_i^{-1}(x))$
where $x \in \cO_{i+1}$. But $\pi_i^{-1}(x) \cong \P^i$ and so this space is of dimension 1 as claimed.
\end{proof}

\excise{
Recall that $f_j$ denotes the inclusion of $\cO_j$ in $G_{n+m}$. Let us write
\[ f_j^* \IC(\overline{\cO_i}) \cong g_{j,i} \cdot \uk_{\cO_j}[nm-j^2] \quad \text{where } p_{j,j} \in \Z[v,v^{-1}].\]
We have that $g_{j,i}$ is zero if $j < i$, is $1$ if $j = i$ and lies in $v^{-1}\Z[v^{-1}]$ if $j > i$. Self-duality gives us an isomorphism
\[ f_j^! \IC(\overline{\cO_i}) \cong \overline{g_{j,i}} \cdot \uk_{\cO_j}[nm-j^2]. \]

Proposition \ref{prop-spectral} then gives us an isomorphism of graded vector spaces
\begin{equation} \label{eq:grsum}
\Hom^{\bullet}(\IC(\overline{\cO_i}), \IC(\overline{\cO_{i^{\prime}}})
\cong \bigoplus_j g_{j,i} g_{j,i^{\prime}} \cdot H^{\bullet}_{P_{m,n} \times P_{n,m}}( \cO_j ). \end{equation}
Recall that we want to calculate
\[ 
\dim
  \operatorname{Ext}^1(\mathbf{IC}(\cO_i),\mathbf{IC}(\cO_{i+
    1}))
= \dim
  \Hom(\mathbf{IC}(\cO_i),\mathbf{IC}(\cO_{i+
    1}[1]))
\]
which is the graded piece of degree 1 in \eqref{eq:grsum} with $i^{\prime} = i+1$. It follows that the above dimension is the coefficient of $v^{-1}$ in $g_{i+1,i}$.

By Proposition \ref{small} we have
\begin{align*}
f_{i+1}^* \IC(\cO_i) & = H^{\bullet}(\P^i) \otimes \uk_{\cO_{i+1}}[nm-i^2] \\
& = H^{\bullet}(\P^i)[2i+1] \otimes \uk_{\cO_{i+1}}[nm-(i+1)^2]. \end{align*}
and so $g_{i+1,i} = v^{-1} + v^{-3} + \dots + v^{-2i-1}$ and the result follows. \end{proof}}

\begin{cor}
  The local chromatographic complex of $j_!\uk_{\cO_0} \langle nm \rangle$ is the unique complex of the form
  \begin{equation*}
    0\to \IC(\cO_0)\to\IC(\cO_1)\langle 1 \rangle \to \cdots \to \IC(\cO_i)\langle i \rangle \to \cdots
  \end{equation*}
  where all differentials are non-zero.  Similarly, that for $j_*\uk_{\cO_0} \langle nm \rangle$, is the unique complex of the form
   \begin{equation*}
     \cdots \to \IC(\cO_i)\langle -i \rangle \to \cdots\to \IC(\cO_1)\langle -1 \rangle \to\IC(\cO_0)\to 0
  \end{equation*}
  also where all differentials are non-zero.
\end{cor}
\begin{remark}
  This corollary shows that this chromatographic complex categorifies
  the MOY expansion of a crossing in terms of trivalent graphs, $\IC(\cO_i)$ corresponding to the MOY graph
  \begin{equation*} 
  \xy 0;/r.5pc/: 
(-6,2)="A"; (6,2)="B";(-4,-2)="C";(4,-2)="D";
 (-15,5); "A"  **\dir{-};  ?(.5)*\dir{>};  ?(.5)+(0,1.5)*{\scriptstyle m};  
 "A";"B" **\dir{-}; ?(.5)*\dir{>}; ?(.5)+(0,1.5)*{\scriptstyle i}; 
 "B";(15,5) **\dir{-}; ?(.5)*\dir{>};  ?(.5)+(0,1.5)*{\scriptstyle n}; 
 (-15,-5);"C"  **\dir{-}; ?(.5)*\dir{>}; ?(.5)+(0,-1.5)*{\scriptstyle n};
"C";"D" **\dir{-}; ?(.5)*\dir{>}; ?(.5)+(0,-1.5)*{\scriptstyle n+m-i}; 
"D";(15,-5) **\dir{-}; ?(.5)*\dir{>};?  ?(.5)+(0,-1.5)*{\scriptstyle m}; 
 "A";"C"  **\dir{-}; ?(.5)*\dir{>};  ?(.55)+(-2.5,0)*{\scriptstyle m-i}; 
 "D";"B"  **\dir{-}; ?(.5)*\dir{>}; ?(.5)+(2,0)*{\scriptstyle n-i};
\endxy
  \end{equation*}
\end{remark}
\begin{proof}
  The terms in the complex are determined by Proposition~\ref{pure-constituents}, and Proposition~\ref{one-d} implies that the isomorphism type of the complex is just determined by which maps are non-zero.
Since $j_!\uk_{\cO_0}$ and $j_*\uk_{\cO_0}$
  are indecomposible, all these maps must be non-zero.
\end{proof}

\section{The invariant via bimodules}
\label{sec:bimodules}

\subsection{The global chromatographic complex of a crossing}
The following lemma gives a description of $\widetilde{\cO_i}$ as a ``Bott-Samelson'' type space:

\begin{lemma}\label{bott-sam-Oi} We have an isomorphism of $P_{m,n} \times P_{n,m}$-equivariant varieties
\begin{equation*}
\widetilde{\cO_i} \cong P_{m,n} \times_{P_{i,m-i,n}} P_{i, m+n-i} \times_{P_{i,n-i,m}} P_{n,m}.
\end{equation*}
\end{lemma}

\begin{proof}
The map sending $[g,h,k]$ to $(gV_i, ghV_n, ghk)$ defines a closed embedding
\begin{equation*}
P_{m,n} \times_{P_{i,m-i,n}} P_{i, m+n-i} \times_{P_{i,n-i,m}} P_{n,m} \hookrightarrow \Gr_i^m \times \Gr_n^{n+m} \times G_{m+n}.
\end{equation*}
Its image is given by triples $(W, V, g)$ satisfying $W \subset V$ and
$V = gV_n$ which is isomorphic to $\widetilde{\cO_i}$ under the map
forgetting $V$. \end{proof}

\begin{defi}
  We let $R_{i_1,\dots,i_m}=\fk[x_1,\dots,x_m]^{S_{i_1}\times \cdots \times S_{i_m}}$.
  be the rings of partially symmetric
  functions corresponding to Young subgroups.  We will use without further mention the canonical isomorphism $
  R_{i_1,\dots,i_m}\cong H^*(BG_{i_1,\dots,i_m})$ sending Chern
  classes of tautological bundles to elementary symmetric functions.
\end{defi}

\begin{cor} As $R_{m,n} \otimes R_{n,m}$-modules, we have a natural isomorphism
\begin{align*}
H^*_{P_{m,n} \times
    P_{n,m}}(\widetilde{\cO_i})&\cong M_i \stackrel{\mathrm{def}}{=} R_{i,m-i,n} \otimes_{R_{i, m + n -i}} R_{i,n-i,m}.\\
\hc_{P_{m,n}\times P_{n,m}}(\IC(\cO_i))&\cong M_i(nm-i^2)
\end{align*}  
\end{cor}
\begin{proof}
  The first equality follows immediately from the main theorem of \cite{BL} (which
  we restated in the most convenient for our work in our earlier paper
  \cite{WW}[Theorem
  3.3]) and Lemma \ref{bott-sam-Oi}.  The second is a consequence of Corollary~\ref{IC=res}. 
\end{proof}

Now have a global version of Proposition~\ref{one-d}:
\begin{prop}
  The spaces of bimodule maps
  \begin{equation*}
\Hom_{R_{m,n} \otimes R_{n,m}}(M_i(-2i),M_{i-1}) \quad\text{ and }\quad\Hom_{R_{m,n} \otimes R_{n,m}}(M_i(2i),M_{i+1})
\end{equation*}
are trivial in degrees $<1$ and one dimensional in degree $1$.
\end{prop}
\begin{proof}
  This follows from \cite[Theorem 5.4.1]{Wil}.  In fact, combined with Proposition \ref{prop-spectral}, the theorem cited above implies that we have isomorphisms 
  \begin{align*}
  \Hom_{R_{m,n} \otimes R_{n,m}}(M_i(-2i),M_{i-1})&\cong \Hom^\bullet\!\left(\mathbf{IC}(\cO_{i}),\mathbf{IC}(\cO_{i-1})\right)\\ 
  \Hom_{R_{m,n} \otimes R_{n,m}}(M_i(2i),M_{i+1})&\cong \Hom^\bullet\!\left(\mathbf{IC}(\cO_i),\mathbf{IC}(\cO_{i+
      1})\right)
\end{align*}
with grading degree on module maps
  matching the homological grading.  Thus, this result is equivalent to Proposition~\ref{one-d}.
\end{proof}

\begin{cor}
  The global chromatographic complex of $j_!\uk_{\cO_0} \langle nm \rangle$ is the unique complex of the form
  \begin{equation}\label{pos-cross}
   \mathbf{M}^-= \cdots\overset{\partial_{i+1}^-}\longrightarrow M_{i+1}(nm-i(i+1))\overset{\partial_i^-}\longrightarrow M_i(nm-i(i-i))\overset{\partial_{i-1}^-}\longrightarrow \cdots 
  \end{equation}
  where all differentials are non-zero.  Similarly, that for $j_*\uk_{\cO_0} \langle nm \rangle$, is the unique complex of the form
  \begin{equation}\label{neg-cross}
     \mathbf{M}^+= \cdots \overset{\partial_{i-1}^+}\longrightarrow M_i(nm-i(1+i))\overset{\partial_i^+}\longrightarrow  M_{i+1}(nm-(i+1)(i+2))\overset{\partial_{i+1}^+}\longrightarrow\cdots
  \end{equation}
  also where all differentials are non-zero.
\end{cor}
We note that these are the complexes defined in \cite[\S 8]{MSV}, with slight change in grading shift, since they have the same modules, and there is only one such complex up to isomorphism.

We note that these maps have a geometric origin.  Consider the correspondence
\begin{equation*}
\widetilde{\cO_{i+1,i}}=\{(U,W,g)\in\Gr^n_{i+1}\times\Gr^n_{i}\times
G_{n+m} \,| gV_n\cap V_m\supset U \supset W\}
\end{equation*}

Obviously, we have natural maps 
\begin{equation*}
  \xymatrix{&\widetilde{\cO_{i+1,i}}\ar[dl]_{p^1_i}\ar[dr]^{p^2_i}&\\\widetilde{\cO_{i+1}}&&\widetilde{\cO_{i}}}
\end{equation*}

\begin{prop}Up to scaling, we have equalities
  \begin{equation*}
    \partial_i^-=(p^2_i)_*(p^1_i)^* \hspace{.4in}\partial_i^+=(p^1_i)_*(p^2_i)^*
  \end{equation*}
\end{prop}
\begin{proof}
  We note that $(p^2_i)_*(p^1_i)^*$ has the expected degree and is
  non-zero.  Thus it must be $\partial_i^-$.  Similarly with
  $(p^1_i)_*(p^2_i)$.
\end{proof}

\subsection{Building the global chromatographic complex I: via canopolis}\label{sec:canopolis}

Now, we are faced with the question of how to build the global
chromatographic complex of an arbitrary braid fragment (by which we
mean a tangle which can be completed to a closed braid by planar
algebra operations).

While the operations we describe are nothing complicated or
mysterious, it can be a bit difficult to both be precise and not pile
on unnecessary notation.  In an effort to give an understandable
account for all readers, we give two similar, but slightly different,
expositions of how to build the complex for a knot, one quite
analogous to Khovanov's exposition in \cite{Kho05} using braids and
their closures, and one in the language of planar algebras and
canopolises, in the vein of the work of Bar-Natan \cite{BN} and the
first author \cite{Web07KR}. 

This approach is based around planar diagrams in sense of planar
algebra; a planar diagram is a crossingless tangle diagram in a planar
disk with holes.  A canopolis is a way of formalizing the process of
building up a tangle by gluing smaller tangles into planar diagrams.

Our definition of our geometric invariant can be phrased in this
language. Given a tangle $T$ written as a union of smaller tangles
$T_i$ in a planar diagram $D$, the space $X_T$ has a product
decomposition $X_T\cong\prod_i X_{T_i}$, and $G_T$ is a subgroup
of $\prod_i G_{T_i}$, given by taking the diagonal inside the
factors corresponding to the edges on $T_i$ and $T_j$
identified by $D$.  

That is, the sheaf $\F_L$ can be built from the sheaves corresponding
to crossings by successive applications of exterior product and
restriction of groups.  It is easy to understand how each of these
affects chromatographic complexes, and our desired invariant can be built piece by piece.

Formally, to each oriented colored tangle diagram in a disk with
boundary points $\{p_1,\dots, p_m\}$, we will associate a complex of
modules over $R_{\Pi}=H^*\left(\prod_{i}BG_{p_i}\right)$, where we use
$\Pi$ to denote all the boundary data of the tangle (the points, their
coloring, their orientation).

The association of the category $\K(R_{\Pi}-\modu)$ of complexes up to
homotopy over $R_{\Pi}$ to the boundary data $\Pi$ (with their
colorings) is a canopolis $\K$, where the functor associated to a planar
diagram is an analogue to that used in the canopolis $\EuScript{M}_0$
in \cite{Web07KR}.  The canopolis functor
\begin{equation*}
\tilde\eta:\K(R_{\Pi_1}-\modu)\times \cdots \times \K(R_{\Pi_k}-\modu)\to \K(R_{\Pi_0}-\modu)
\end{equation*}
associated to a planar diagram with outer circle labeled with $\Pi_0$
and $k$ inner circles labeled with $\Pi_1,\dots,\Pi_k$ will be given
by tensoring with a complex of  $R_{\Pi_0}$-$R_{\Pi_1}\otimes\cdots
\otimes R_{\Pi_k}$ bimodules.  We let $R_{\Pi_*}=R_{\Pi_1}\otimes\cdots
\otimes R_{\Pi_k}$

Let $\mc A(\eta)$ be the set of arcs in $\eta$, and let $\al_a,\om_a$
be the tail and head of $a\in \mc A(\eta)$, and let $n_a$ be the
integer $a$ is colored with. Associated to each arc, we associate the
sequence
\begin{equation*}
(e_1({\om_a})-e_1({\al_a}), \dots , e_{n_a}(\om_a)-e_{n_a}(\al_a)),
\end{equation*}
which identifies the classes $e_i\in H^*(BG_n)$ corresponding to the
elementary symmetric polynomials (geometrically, these are the Chern
classes of the tautological bundle on $BG_n$) for the endpoints
connected by the arc. To our diagram, we associate the concatenation
of these sequences.

Let $\kappa(\eta)$ be the Koszul complex over $R_{\Pi_0}\otimes\cdots
\otimes R_{\Pi_k}$ of this concatenated sequence for our diagram
$\eta$, which we think of as a bimodule with the $R_{\Pi_0}$-action on
the left and the $R_{\Pi_*}$ on the right.
\begin{defi}
  The canopolis functor $\tilde \eta$ associated to the diagram $\eta$ is $
\kappa(\eta)\otimes_{R_{\Pi_*}}-$.
\end{defi}

\begin{prop}
  The map sending a tangle $T$ to the global chromatographic complex
  of $\F_T$ is a canopolis map.
\end{prop}
\begin{proof}
  We simply need to justify why tensoring with such a Koszul
  resolution (which is a free resolution of the diagonal bimodule for
  $H^*(BG_{p_i})$) is the same as changing $G_T$ to only include the
  diagonal subgroup of $G_{\om_a}\times G_{\al_a}$.  This is one of
  the basic results of \cite{BL} (as we mentioned earlier, this is
  rephrased most conveniently for us in \cite[Theorem 3.3]{WW}).
\end{proof}

\begin{remark}
  We note that this construction at no point used the fact that our
  diagram should be a braid fragment; unfortunately, it is unclear
  whether our construction will be invariant under the oppositely
  oriented Reidemeister II move, as with Khovanov-Rozansky's original
  construction (see, for example, \cite[\S 3]{Web07KR}) though we will
  note that proving invariance under this move for the all 1's
  labeling is sufficient to imply it for all labeling, by the same
  cabling arguments we will use later.
\end{remark}

\subsection{Building the global chromatographic complex II: via bimodules}
\label{sec:build-glob-chrom}

A less flexible, but perhaps more familiar, perspective is to
associate to each braid a complex of bimodules, in a manner similar to
\cite{Kho05} (though the same complex had previously appeared in other
works on geometric representation theory).  In the case where all
labels are 1, our construction will coincide with Khovanov's.

As in Section \ref{sec:braids-sheav-groups}, we let $\be$ be a braid
with $n$ strands, and $\Bn=(i_1,\dots, i_m)$ be the labels of the
top end of the strands (so $\beBn$ is the labeling of the bottom
end).  In that section, we showed the our invariant can also be
described in terms of the chromatographic complex of a sheaf $\fF_\be$
on $G_N$.

This sheaf has the advantage that it can be built from the sheaves for
smaller braids by convolution of sheaves.  However, convolution of
sheaves is a geometric operation which is not always easy to
understand.  Thus, we will give a description of it using tensor
product of bimodules.  Let $F(\be)$ be the $P_\Bn\times
P_{\beBn}$-equivariant global chromatographic complex of $\fF_\be$,
considered as a complex of bimodules over $H^*(BP_{\Bn})$ and
$H^*(BP_{\beBn})$.

\begin{prop}
  We have natural isomorphisms
  \begin{equation*}
    F(\be\be')\cong F(\be)\otimes_{H^*(BP_{\beBn})}F(\be').
  \end{equation*}
\end{prop}
\begin{proof}
  Consider the exterior product $\fF_{\be}\boxtimes\fF_{\be'}$ on $G_N\times G_N$.  The $P_\Bn\times P_{\beBn}\times P_{\beBn} \times
  P_{\bepBn}$-equivariant chromatographic complex of this is
  $F(\be)\otimes_\C F(\be')$.  If we restrict to the diagonal
  $P_{\beBn}$, then this complex is
  $F(\be)\overset{L}\otimes_{H^*(BP_{\beBn})} F(\be')$.  By the
  equivariant formality of all simple, Schubert-smooth perverse
  sheaves on a partial flag variety, $F(\be)$ is free as a right
  module, so it is not necessary to take derived tensor product.

  By the convolution description, we have
  \begin{equation*}
  \fF_{\be'\!\be}\cong {}_{P_\Bn\times P_{\beBn}\times 
  P_{\bepBn}}^{P_{\Bn}\times
  P_{\bepBn}}\mu_*(\fF_{\be,\be'})
\end{equation*}
where
  $\mu:G_N\times G_N\to G_N$. Since $G/P_{\beBn}$ is projective, this map
  simply has the effect of forgetting the $H^*(BP_{\beBn})$ action on each page of the chromatographic spectral sequence.
\end{proof}

Thus, we can construct $F(\be)$ just by knowing the complex
$F(\si_i^{\pm 1})$ for the elementary twists $\si_i^{\pm 1}$.
However, first we must compute the corresponding sheaves.  Given
$\Bn$, we let $Q_j=P_{i_1,\ldots,i_j+i_{j+1},\ldots,i_n}$, and let
$\becircled{Q}_j=Q_j-Q_0$.
\begin{prop}
  We have isomorphisms
  \begin{equation*}
\fF_{\si_i}=j_*\uk_{\becircled Q_i}\langle i_ii_{i+1} \rangle\hspace{.5in}\fF_{\si_i^{-1}}=j_!\uk_{\becircled Q_i}\langle i_ii_{i+1} \rangle,
\end{equation*}
where $j:\becircled
  Q_i\hookrightarrow G_N$ is the obvious inclusion.
\end{prop}

The global complex of this is
very close to the complex $\bM^+$ described in (\ref{pos-cross}),
considered as a complex of $R_{i_i,i_{i+1}}$-$R_{i_{i+1},i_i}$
bimodules.  However, we must extend scalars to get a complex of
$R_{\Bn}$-$R_{\si_i\Bn}$ bimodules
\begin{prop}
$\displaystyle{F(\si_i^{\pm 1})=R_{i_1,\dots,i_{i-1}}\otimes_\Q \bM^{\pm}\otimes_\Q R_{i_{i+2},\dots,i_k}.}$
\end{prop}

Again, this is precisely the complex given in \cite[\S 8]{MSV} up to grading shift.

If $\beBn=\Bn$, then we can close this braid to a link.  Our
definition of the knot invariant for this link is the equivariant
chromatographic complex for the diagonal $P_\Bn$-action.  By the
authors' previous work \cite[Theorem 1.2]{WW}, this coincides with the
Hochschild homology $\HH^*(F(\be))$, applied termwise of the complex
$F(\be)$.

\begin{prop}
  The cohomology of the complex $\HH^*_{R_{\Bn}}(F(\be))$ coincides with the
  invariant $\KM_2(\clo{\be})$ of the closure of the braid.

  In fact, the chromatographic spectral sequence is exactly the natural spectral sequence
  \begin{equation*}
    \H^i\!\left(\HH^j(F(\be))\right)\Rightarrow \H^{i+j}(R_{\Bn}\overset{L}\otimes_{R_{\Bn}\otimes R_{\Bn}}F(\be)).
  \end{equation*}
\end{prop}
\begin{proof}
  Let $\pi:G_N \to pt$, and consider the object $\pi_*\fF_\be$ in
  the equivariant derived category $D_{P_{\Bn}\times P_{\Bn}}(pt)$.
  Under the equivalence to $R_{\Bn}$-dg-bimodules given in
  \cite[Theorem 7]{WWequ}, this is sent to the complex $F(\si)$.
  Similarly, the weight filtration is sent to that induced by thinking
  of $F(\be)$ as a complex.  Thus, the spectral sequences match under
  this equivalence.
\end{proof}

Since $\H^*\!\left(\HH^*(F(\be))\right)$ is precisely the invariant proposed by
\cite{MSV}, Theorem \ref{comparison} follows immediately.

\section{Decategorification}
\label{sec:decategorification}

We also wish to show that our knot invariant is, in fact, a
categorification of the HOMFLYPT polynomial.  

\subsection{A categorification of the Hecke algebra}
\label{sec:categ-hecke-algebra}

This requires a few basic
results about the relationship between sheaves on $G_n$ and the Hecke
algebra $\Hec_n$.  As usual, $B=P_{1,\dots,1}$ is the standard Borel.
\begin{defi}
  The Hecke algebra $\Hec_n$ is the algebra over $\Z[q^{\nicefrac{1}{2}},q^{\nicefrac{-1}{2}}]$ given by the quotient of the group algebra of the braid group $\Br_n$ by the quadratic relation
  \begin{equation*}
    (\si_i+q^{\nicefrac{1}{2}})(\si_i-q^{\nicefrac{-1}{2}})=0
  \end{equation*}
  for each elementary twist $\si_i$.
\end{defi}

\begin{prop}[\cite{KW}]
  The Grothendieck group $K^0\hspace{-.8mm}\left(D^b_{B\times B}(G_n)\right)$ of the equivariant derived category $D^b_{B\times B}(G_n)$ is isomorphic to the Hecke algebra $\Hec_n$, with the convolution product decategorifying to the algebra product in $\Hec_n$.

This map is fixed by the assignment $$[j_*\uk_{Bs_iB}]\mapsto q^{\nicefrac{1}{2}}\si_i$$ where $j:Bs_iB\hookrightarrow G_n$ is the obvious inclusion.
\end{prop}

Let $\F$ be a $\BB$-equivariant sheaf on $G_n$.  Then we have a map 
\begin{equation*}
  \Eul_{B}(G;\F)=\sum_{i,j,k}(-1)^{\ell}q^{\nicefrac{j}{2}}t^k\dim\mathbb{H}^{j-\ell;{j-k}}_{B_\Delta}(\gr^W_\ell\F)
\end{equation*}
sending the class of $\F$ in the Grothendieck group to the bi-graded
Euler characteristic of its global chromatographic complex, often
called the {\bf mixed Hodge polynomial}.

This map agrees with a previously known trace on the Hecke algebra, a
fact that the authors have proven in a separate note, due to its
independent interest and separate connection to the question of
constructing Markov traces on general Hecke algebras.

\begin{prop}{\cite[Theorem 1]{WWmar}}\label{JO}
The map $\Eul_B(G_n;-)$ is the Jones-Ocneanu trace $\Tr$ \cite{Jon87} on $\Hec_n$ with appropriate normalization factors.
\end{prop}
\begin{remark}
  This geometric definition applies equally well to any simple Lie
  group, and defines a canonical trace on the Hecke algebra for any
  type.  In fact, our construction can be modified in a
  straightforward way to a ``triply graded homology'' invariant on all
  Artin braid groups.  In type B, this can be interpreted as a
  homological knot invariant for knots in the complement of a torus.
\end{remark}

\subsection{Decategorification for colored HOMFLYPT}
\label{sec:decat-color-homflypt}

To apply this result, we must relate our construction to the
categorification of the Hecke algebra above.  Recall that if $\si$ is
a braid labeled all with 1's, then $\fF_\si$ is an object of
$D^b_{B\times B}(G_n)$

\begin{prop}
  The class $[\fF_{\si}]\in\Hec_n$ is the image of $\si$ under the
  natural map $\Br_n\to \Hec_n$.
\end{prop}
This, combined with Proposition \ref{JO}, gives a new proof of the result of Khovanov \cite{Kho05} that all components are
labeled with 1, the invariant
$$\Eul(L)=\Eul_{G_L}(X_L;\F_L)=\sum_{i,j,k}(-1)^{\ell}q^jt^k\dim
\KM_2^{j;k;\ell}(L)$$ 
is the appropriately normalized HOMFLYPT polynomial of $L$.  
We wish to extend this to the colored case.  For this, we must use a ``cabling/projection'' formula.

Consider a closable colored braid $\si$, and let $P=P_\Bn$ and $G=G_N$. We have defined a $P\times P$-equivariant sheaf $\fF_\si$ on $G$ by the multiplication map $m:X_\si\to G$.  

\begin{thm}\label{decat}
  For any colored link $L$, the Euler characteristic $\Eul(L)$ is the (suitably normalized) colored HOMFLYPT polynomial.
\end{thm}

In order to prepare for the proof, we show a pair of lemmata.  Let $\si_{cab}$ denote the cabling of $\si$ in the blackboard framing with multiplicities given by the colorings, thought of as colored with all 1's.
  \begin{lemma}\label{cable}
    We have an isomorphism of $P\times B$-equivariant sheaves
    \begin{equation*}
\res^{P\times P}_{P\times B}\fF_\si\cong \ind^{P\times B}_{B\times B}\fF_{\si_{cab}}.
\end{equation*}
\end{lemma}
\begin{proof}
  The proof is a straightforward induction on the length of $\si$; left to the reader.
\end{proof}

Let $\la_\Bn$ be the partition given by arranging the parts of $\Bn$
in decreasing order, and let $\la_\Bn^t$ be its transpose.  Let
$\pi_\Bn$ be the projection in the Hecke
algebra to the representations indexed by Young diagrams less than
$\la_\Bn^t$ in dominance order.  Alternatively, if we identify
$\Hec_N$ with the endomorphisms of $V^{\otimes N}$ where $V$ is the
standard representation of $U_q(\mathfrak{sl}_m)$ for $m\geq n$, then
this is the projection to $\wedge^{i_1}V\otimes \cdots \otimes
\wedge^{i_n}V$.

Let $q_P=\sum_{W_P}q^{\ell(w)}$ be the Poincar\'e polynomial of the flag variety $P/B$. 

\begin{lemma}\label{proj}
  We have $[\res_{P\times B}^{B\times B}\ind^{P\times B}_{B\times B}\fF]=q_P\pi_P[\fF]$.
\end{lemma}
\begin{proof}
  First consider the case where $P=G$.  In this case, the sheaf $\res_{G\times
    B}^{B\times B}\ind^{G\times B}_{B\times B}\fF$ has a filtration
  whose successive quotients are of the form $\mathbb{H}^i(\fF)\otimes \uk_G$.  Thus
  we have
  \begin{equation*}
[\res_{G\times B}^{B\times B}\ind^{G\times B}_{B\times
    B}\fF]=\dim_q\hc(\fF)\cdot[\uk_{G}].
\end{equation*}
It is a classical fact that
  $[\uk_G]=q_G\pi_G$; here $\pi_G$ is just the projection to
  $\wedge^NV$.  This computation immediately extends to the general
  case.
\end{proof}
\begin{remark}
  This proposition shows why our approach works for colored HOMFLYPT
  polynomials, but would need to be modified to approach the HOMFLY
  polynomials for more general type A representations; we lack a good
  categorification of most of the projections in the Hecke algebra,
  but $\pi_P$ has a beautiful geometric counterpart.  This may be
  related to the fact that $\pi_P$ is the projection not just to a
  subrepresentation, but in fact to a cellular ideal in $\Hec_n$.
\end{remark}

\begin{proof}[Proof of Theorem~\ref{decat}]
Immediately from Lemmata \ref{cable} and \ref{proj}, we have the equality of Grothendieck classes $[\res^{P\times P}_{B\times B}\fF_\si]=q_P\pi_{P}[\fF_{\si_{cab}}]$.  Thus
\begin{align*}
  \Eul_P(G;\fF_\si)&=q^{-1}_P\Eul_B(G;\res^{P\times P}_{B\times B}\fF_\si)\\
     &=\Tr (q^{-1}_P[\res^{P\times P}_{B\times B}\fF_\si])\\
     &=\Tr (\pi_P[\fF_{\si_{cab}}])
\end{align*}
By the ``projection/cabling'' formula (see, for example, \cite[Lemma 3.3]{LZ}), this is precisely the colored HOMFLYPT polynomial.
\end{proof}

\section{The proof of invariance: $\GL2$}
\label{sec:invariant}

We first concentrate on the simpler case of $\GL2$ before attacking
the general case. In this case, we will obtain an invariant which
matches the HOMFLYPT homology of Khovanov-Rozansky \cite{KR05,Kho05},
so the section below can be thought of as a geometric proof of the
invariance of this homology theory.

Recall that if $\sigma$ is a braidlike diagram on $n$ strands
we described in Section \ref{sec:braids-sheav-groups} a map
\[
m : X_{\sigma} \to G_n
\] equivariant with respect to $\phi: G_{\sigma} \to T \times T$, where $T \times T$ acts on $G_n$ by left and right multiplication.
This map gives rise to a functor
\[ {}_{G_{\sigma}}^{B \times B} m_* : D^+_{G_{\Gamma}}(X_{\Gamma}) \to D^+_{T \times T}(G_n) \]
and we denoted the image of $\F_{\sigma}$ by $\fF_{\sigma}$. We saw
that this functor preserves weight filtrations.

Now suppose that $w$ is an element of the symmetric group on
$n$-letters (which we regard as permutation matrices in $G_n$) and
that $\sigma = \sigma_{i_1} \sigma_{i_2} \dots \sigma_{i_p}$ is a (positive)
braid in the standard generators corresponding to a reduced expression
$s_{i_1} \dots s_{i_p}$ for $w$.

It is straightforward to see that if we restrict $m$ to the open set
$\tilde{U}$ in $G_{\Gamma}$ consisting of tuples $(g_1, \dots , g_p)$
where each $g_i \in U$ (where $U$ denotes the open Bruhat cell
in $G_2$) then we may factor $m$ as
\begin{equation}
\label{map:RIIfactor}  \tilde{U} \to \tilde{U}/ \ker \phi \to G_n
\end{equation}
where the first map is a quotient by a free action, and the second map 
is an isomorphism.

Moreover, if we denote by $B$ the subgroup of upper triangular
matrices, then the image of the restriction of $m$ to $\tilde{U}$ is
contained in Schubert cell $BwB$. It follows that
\[ \Phi_{\sigma} = {j_w}_! \uk_{BwB}\langle \ell(w) \rangle. \]
(Here $j_w$ denotes the inclusion of the Bruhat cell $BwB$ into $G_n$).

\begin{prop}
  Theorem \ref{invariance} holds in the case where all strands are
  labeled by 1.
\end{prop}
\begin{proof}
As usual with proofs that knot invariants defined in terms of a
projection are really invariants, we check that our description is
unchanged by the Reidemeister moves.  Since we only consider closed
braids, we only need to check Reidemeister II and III in the
braid-like case, when all strands are coherently oriented. Those who
prefer to use the Markov theorem can consider the proof of
Reidemeister I as a proof of the Markov 1 move, and the Reidemeister
II and III calculations as proving the independence of the
presentation of our braid in terms of elementary twists {\it and} of
the Markov 2 move (which only uses Reidemeister IIa).

In each case, we will use the fact that while we wish to compare the
pushforwards of sheaves corresponding to diagrams $L$ and $L^{\prime}$ on 
from $X_L/G_L$ and $X_{L^{\prime}}/G_{L^{\prime}}$ to
a point, we can accomplish this by showing that their pushforwards by
any pair of maps to any common space coincide.  Being able to
use these techniques is one of the principal advantages of a
geometric definition over a purely algebraic one.

\emph{Reidemeister I}: Consider the following tangles:
\begin{equation}\label{reid1}
D = \begin{array}{c} \reflectbox{{\includegraphics[totalheight=2.2cm]{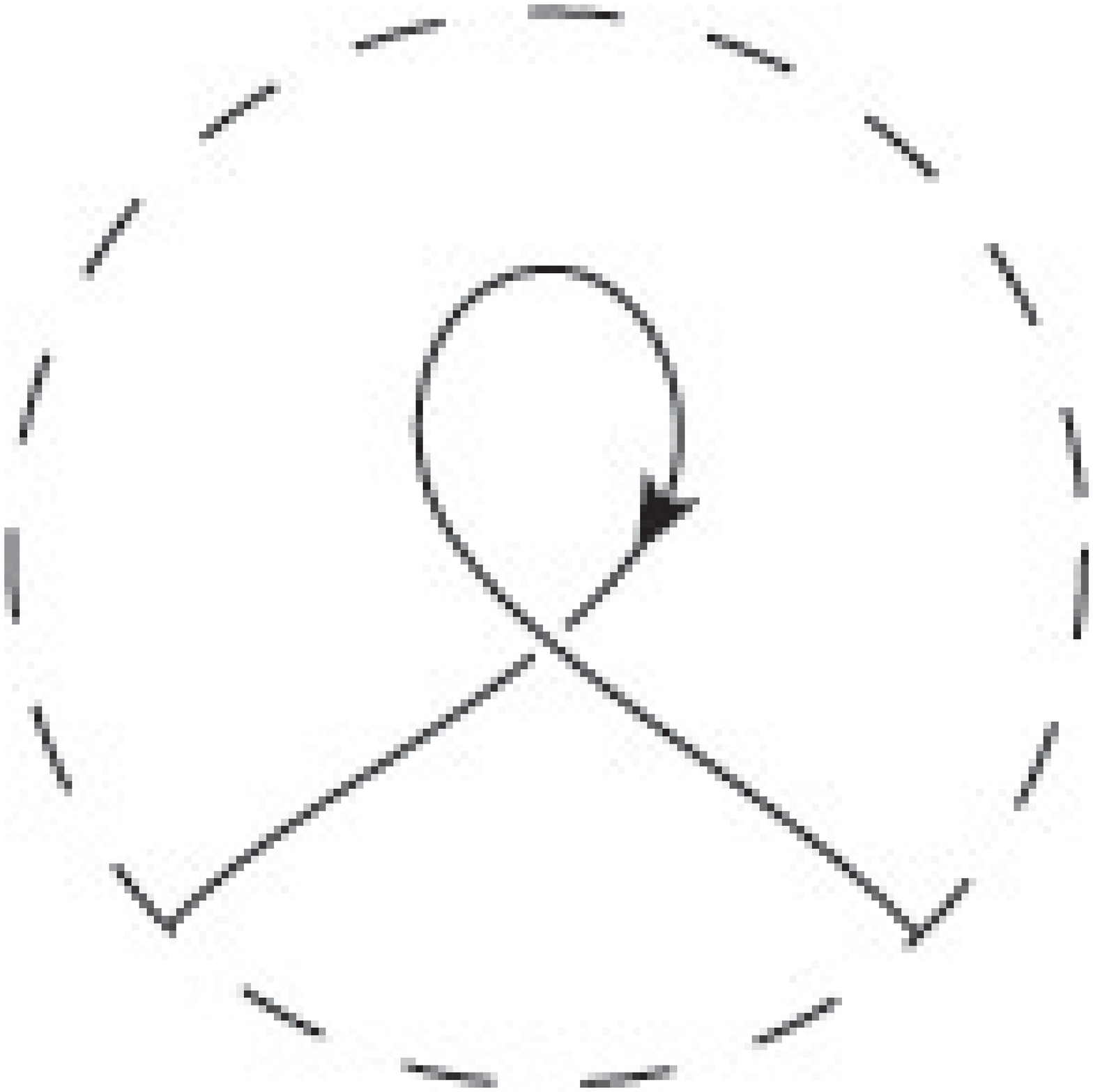} }}
\end{array}
\qquad D^{\prime} =  \begin{array}{c} \reflectbox{{\includegraphics[totalheight=2.2cm]{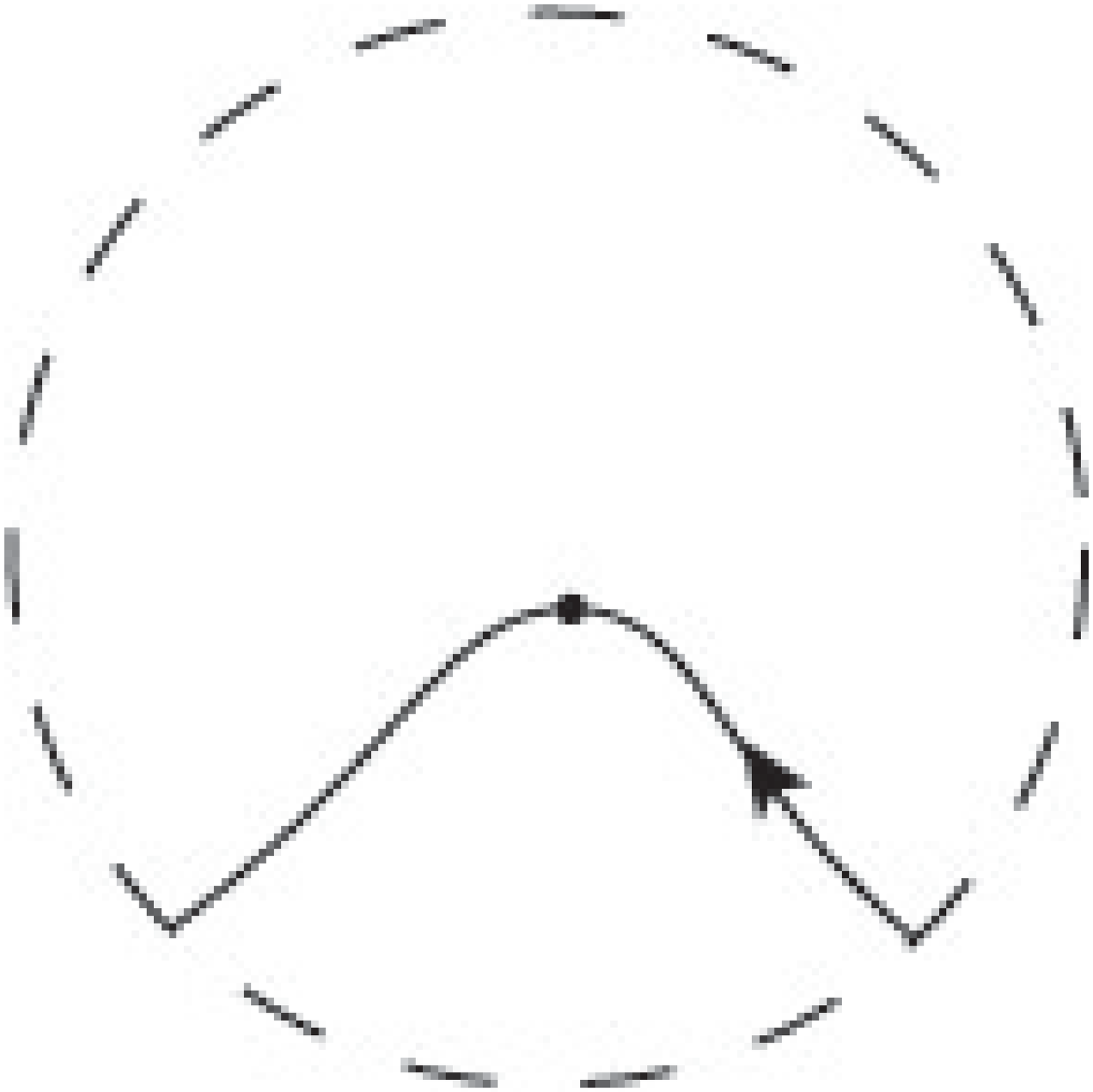} }}\end{array}.
\end{equation}
To simplify notation we denote the associated varieties $X$,
$X^{\prime}$ and groups $G$, $G^{\prime}$ respectively. We have $X =
G_2$ and $X^{\prime} = G_1$, $G = G_1^3$ and $G^{\prime} = G_1^2$. The
determinant gives a map
\begin{equation*}
d : X \to X^{\prime}
\end{equation*}
which is equivariant with respect to the map $\phi : G \to G^{\prime}$
forgetting the factor corresponding to the internal edge. Reidemeister I will result from an
isomorphism
\begin{equation*}
{}^{G^{\prime}}_G d_* \F_D \cong \F_{D^{\prime}}
\end{equation*}
compatible with the weight filtrations on both sheaves. Note that the
weight filtration on $\F_{D^{\prime}}$ is trivial, whereas that on
$\F_D$ is not.

Let $B \stackrel{a}{\hookrightarrow} X \stackrel{b}{\hookleftarrow}
BsB$ be the decomposition of $X = G_2$ into its two Bruhat cells. We
have an distinguished triangle
\[ a_!a^! \uk_X \langle 1 \rangle \to \uk_X\langle 1 \rangle \to b_*b^* \uk_X\langle 1 \rangle \triright \]
turning the triangle gives the weight filtration on $b_*\uk_{BsB}\langle 1 \rangle$:
\begin{equation} \label{eq:triRI} \uk_X\langle 1 \rangle \to b_* \uk_{BsB}\langle 1 \rangle \to
a_*\uk_B(-\nicefrac{1}{2})
\triright . \end{equation}
In the following we analyze the effect of ${}^{G^{\prime}}_Gd_*$ on this triangle.

The restriction of $d$ to $BsB \subset X$ is a trivial $G_1 \times
\bA^2$-bundle over $X^{\prime}$. One may easily check that $\ker \phi$
acts freely on the multiplicative group in the fiber. It follows that
\[ {}^{G^{\prime}}_Gd_* b_* \uk_{BsB} \cong \uk_{X^{\prime}}. \] On
the other hand, the restriction of $d$ to $B \subset X$ yields a
trivial $G_1 \times \bA^1$ bundle, with $\ker \phi$ only acting on
$\bA^1$. It follows that
\[ {}^{G^{\prime}}_Gd_* a_* \uk_B = H^{\bullet}(\P^{\infty}) \otimes H^{\bullet}(G_1) \otimes \uk_{X^{\prime}}. \] 

Applying ${}^{G^{\prime}}_Gd_*$ to \eqref{eq:triRI} and using the above isomorphisms we obtain
\begin{equation*} {}^{G^{\prime}}_G d_* \uk_{X}\langle 1 \rangle \to
  \uk_{X^{\prime}} \langle 1 \rangle \to H^{\bullet}(\P^{\infty})
  \otimes H^{\bullet}(G_1) \otimes \uk_{X^{\prime}}(-\nicefrac{1}{2})
  \triright. \end{equation*} As $\Hom(\uk_{X^{\prime}},
\uk_{X^{\prime}} [i]) = H^i_{G^{\prime}} (X^{\prime})$ is zero for $i
< 0$ we conclude that the second arrow above is zero. Hence the
filtration on $\uk_{X^{\prime}}$ may be taken to be trivial (and
therefore agrees with that on $\F_{D^{\prime}}$ up to $\langle 1
\rangle$).

\emph{Reidemeister IIa}:
Here we are concerned with the two tangles:
\begin{equation*}
D = \begin{array}{c} \reflectbox{
{\includegraphics[totalheight=2.2cm]{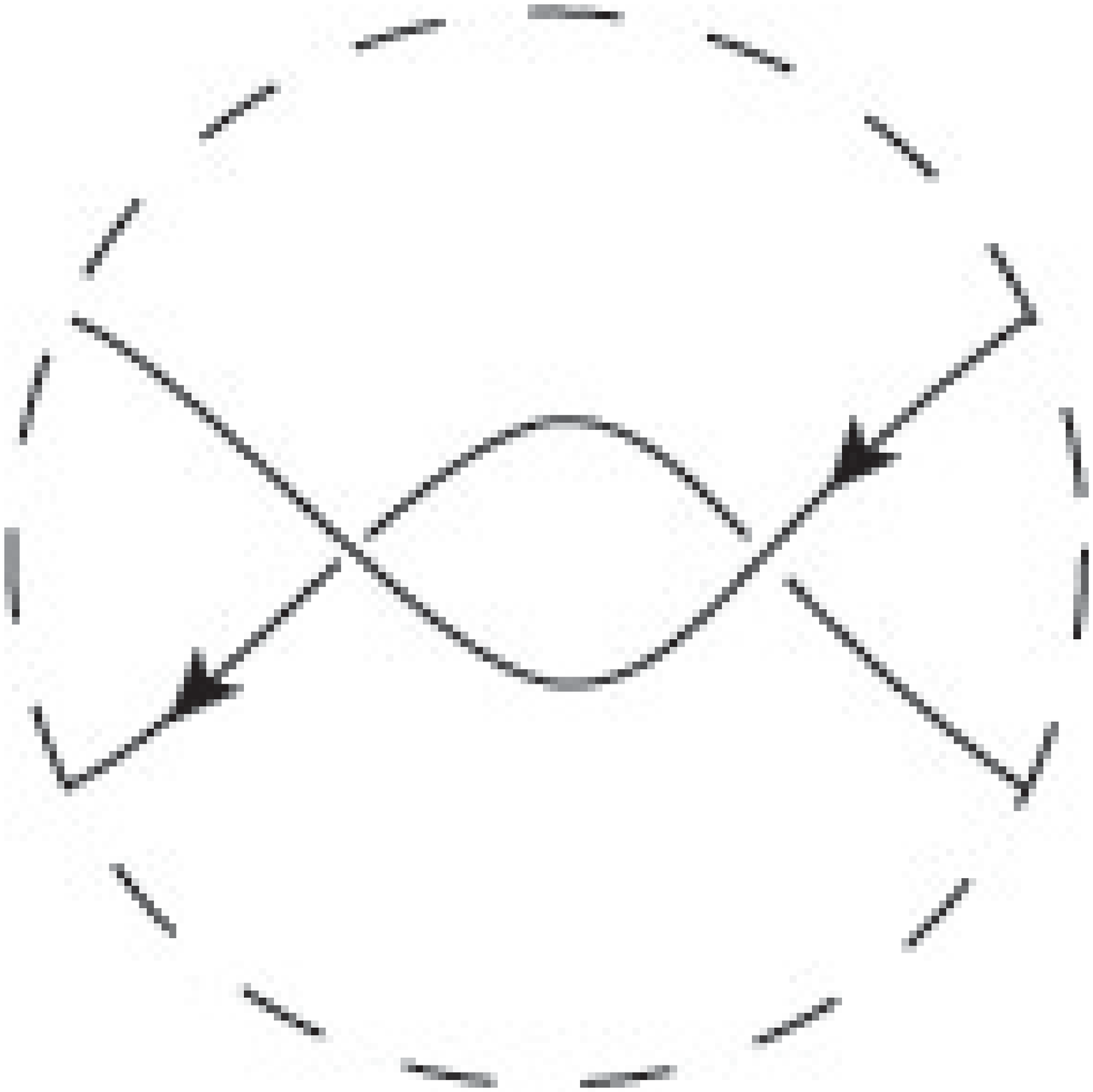}} }\end{array}
\qquad D^{\prime} =  \begin{array}{c} \reflectbox{{\includegraphics[totalheight=2.2cm]{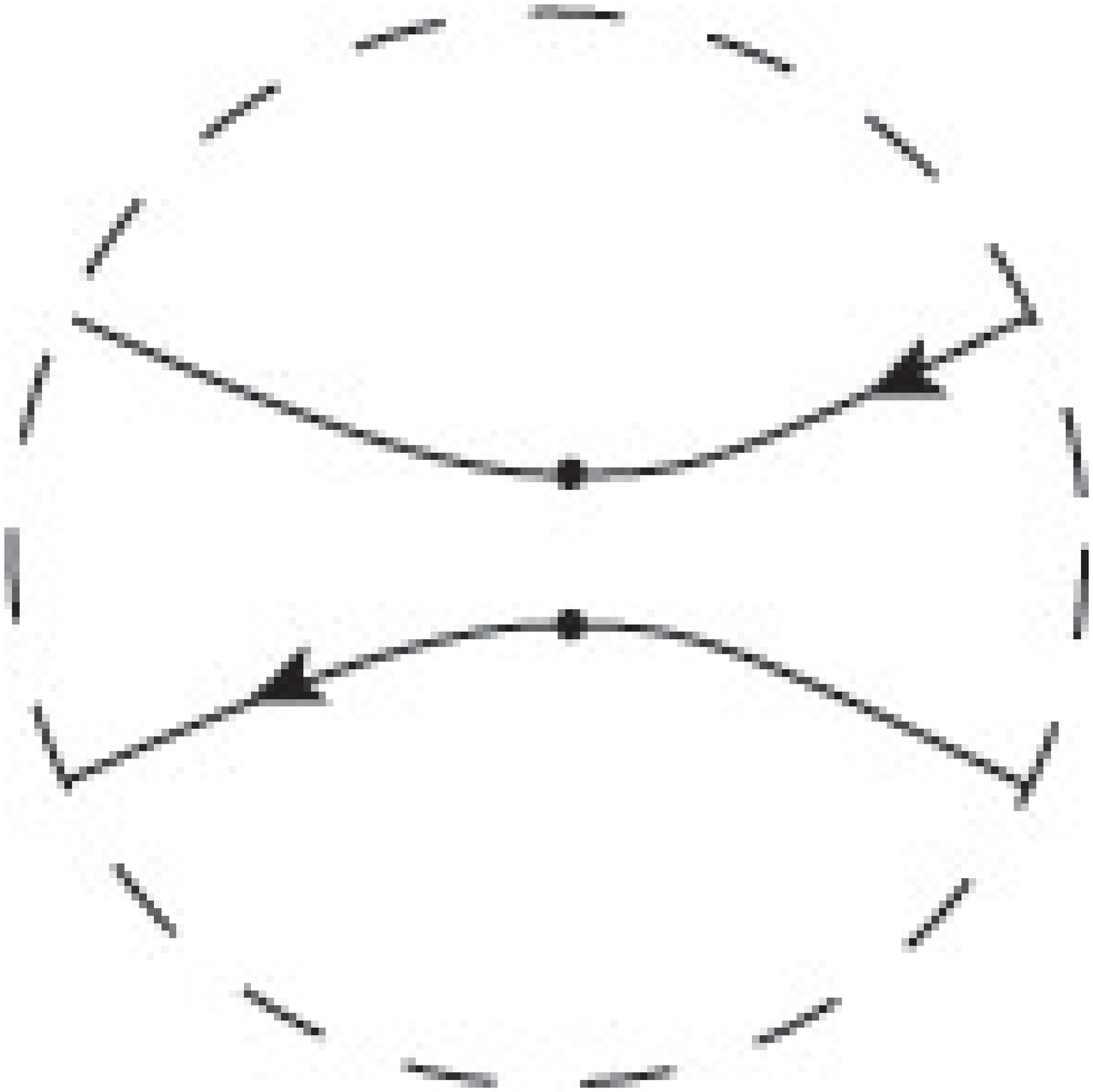}} }\end{array}.
\end{equation*}
We denote the associated varieties and groups $X, X^{\prime}, G,
G^{\prime}$.  We denote by $m$ the multiplication map $X \to G_2$
considered at the start of this section. We regard $X^{\prime}$ as the diagonal matrices inside
$G_2$.

We have seen that ${}^{G^{\prime}}_Gm_*$ preserves weight filtrations, and hence we may ignore weight filtrations when comparing ${}^{G^{\prime}}_Gm_* \F_D$ and $\F_{D^{\prime}}$. The map $B \to X^{\prime}$ forgetting the off-diagonal entry is acyclic, and therefore it is enough to show that ${}^{G^{\prime}}_Gm_* \F_D \cong \uk_B$.

We decompose $G_2$ into its Bruhat cells $B \stackrel{a}{\hookrightarrow} G_2 \stackrel{b}{\hookleftarrow} BsB$ as before. We claim we have isomorphisms:
\begin{align} \label{RIIa}
{}^{G^{\prime}}_Gm_*(a_* \uk_B \boxtimes b_! \uk_{BsB}) & \cong b_! \uk_{BsB} \\
\label{RIIb} {}^{G^{\prime}}_Gm_*(\uk_G \boxtimes a_* \uk_B) & \cong \uk_G \\
\label{RIIc} {}^{G^{\prime}}_Gm_*(\uk_G \boxtimes \uk_G) & \cong \uk_G \oplus \uk_G\langle -2 \rangle \\
\label{RIId} {}^{G^{\prime}}_Gm_*(\uk_G \boxtimes b_! \uk_{BsB}) & \cong \uk_G \langle -2 \rangle
\end{align}
(As always we regard the exterior tensor product of equivariant sheaves on $G_2$ as an equivariant sheaf on $X$ via restriction.)

Indeed, \eqref{RIIa} and \eqref{RIIb} follow from the fact that the
restriction of $m$ to $B \times G$ or $G \times B$ is a trivial
$B$-bundle, with $\ker \phi$ acting freely on the multiplicative
groups in the fiber. The factorization \eqref{map:RIIfactor} of $m$ as
``essentially a $\P^1$-bundle'' implies \eqref{RIIc}. Then
\eqref{RIId} follows from the others by taking the exterior tensor
product of $\uk_G$ with the distinguished triangle $b_! \uk_{BsB} \to
\uk_G \to a_* \uk_B \to$ and applying ${}^{G^{\prime}}_Gm_*$.

Now $B$ is smooth of codimension 1 inside $G_2$ so $a^! \uk_G = \uk_B
\langle -2 \rangle$ and we have an exact triangle
\[ a_* \uk_B \langle -2 \rangle \to \uk_G \to b_* \uk_{BsB} \triright. \]
Taking the exterior tensor product with $b_!\uk_{BsB}$,  applying ${}^{G^{\prime}}_G m_*$ and using the above isomorphisms we obtain a distinguished triangle
\begin{equation} \label{RII:tri}
 b_! \uk_{BsB}\langle -2 \rangle \to \uk_G \langle -2 \rangle \to 
{}^{G^{\prime}}_Gm_*(b_* \uk_{BsB} \boxtimes b_! \uk_{BsB} ) \triright \end{equation}
Note that $\Hom( b_! \uk_{BsB}, \uk_G)$ is one dimensional and contains the adjunction morphism $b_!b^! \uk_{G} \to \uk_G$. By considering its dual, one may show that the first arrow in \eqref{RII:tri} is non-zero. It follows that this arrow is the adjunction morphism (up to a non-zero scalar) and we have an isomorphism:
\[ {}^{G^{\prime}}_Gm_*(b_* \uk_{BsB} \boxtimes b_! \uk_{BsB} ) \cong \uk_B \langle -2 \rangle \]
Finally note that by definition $\F_D$ is $b_* \uk_{BsB} \boxtimes b_! \uk_{BsB} \langle 2 \rangle$ and so
 \[ {}^{G^{\prime}}_Gm_* \F_D \cong \uk_B \]
which finishes the proof of invariance under Reidemeister II.

\emph{Reidemeister III}: This follows immediately from the considerations at the beginning of this section. Indeed, if $\sigma$ and $\sigma^{\prime}$ are the diagrams corresponding to the words $\sigma_1 \sigma_2 \sigma_1$ and $\sigma_2\sigma_1 \sigma_2$ we have maps
\[ X_{\sigma} \stackrel{m}{\to} G_3 \stackrel{m^{\prime}}{\leftarrow} X_{\sigma^{\prime}} \]
and
\[{}_{G_{\sigma}}^{T \times T} m_* \F_{\sigma} 
\cong j_{w_0} \uk_{Bw_0B} \cong
 {}_{G_{\sigma}}^{T \times T}m^{\prime}_* \F_{\sigma^{\prime}} \]
(here $w_0$ indicates the longest element in $S_3$).
\end{proof}

\section{The proof of invariance: $\GL n$}
\label{sec:proof-invariance:-gl}

Now, we expand to the full case of all possible positive integer labels.

\begin{proof}[Proof of Theorem \ref{invariance}]
  All of the Reidemeister moves can simply be reduces to the
  corresponding statement for the cabling with the all 1's labeling.
  Interestingly, the same trick was used in \cite{MSV} to prove
  invariance in a special case.  Almost certainly our proof could be
  rephrased in a purely algebraic language like their paper, though at
  the moment it is unclear how.

  \emph{Reidemeister IIa \& III}:\, Here we need only establish the
  isomorphisms of $P\times P$-equivariant sheaves
\begin{equation*}
  \fF_{\si_i}\star\fF_{\si_i^{-1}}\cong \uk_{P}\hspace{.5in}
  \fF_{\si_i}\star\fF_{\si_{i+1}}\star\fF_{\si_i}\cong\fF_{\si_{i+1}}\star\fF_{\si_i}\star\fF_{\si_{i+1}}
\end{equation*}
Lemma~\ref{cable} implies that these hold as $P\times B$ equivariant
sheaves, applying the invariance for the all 1's labeling to the
cable.  

In fact, both are the $*$-inclusion of a local system on a $P\times
P$-orbit: $P$ itself in first case, the $P\times P$ orbit of the
permutation corresponding to the cabling of $\si_i\si_{i+1}\si_i$.
Since the stabilizer of any point under $P\times P$ is connected, any
$P\times B$ equivariant local system on an orbit has at most one
$P\times P$ equivariant structure, and this equality holds as $P\times
P$ equivariant sheaves.

{\em Reidemeister I:} We again use the ``cabling/projection''
philosophy, but this argument requires a bit more subtlety.  We are
interested in the chromatographic complex of a single crossing with
its right ends capped off, that is, the tangle projection denoted by
$D$ in (\ref{reid1}).  To construct the sheaf $\F_D$, we take
$U\subset G_{2n}$, as defined in (\ref{u-def}), and consider
$j_*\uk_U$ or $j_!\uk_U$, depending on whether our crossing is
positive or negative.  These cases are Verdier dual, and the proofs of
invariance are essentially identical, so we will treat the positive
case, and only note where the negative differs.

We consider the action on $G_{2n}$ of $G_{n,n}$ on the left {\it and}
the right.  By convention, we let $G_n^1$ denote the first copy of
$G_n\subset G_{n,n}$ and $G^2_n$ the second.  As before, we let $T_n$
be diagonal matrices in $G_n$, and we use $T^1_n,T^2_n$ for the
inclusions into the two factors.  We let $G^{1,1,2}_{n,n,n}$ denote
$G_n^1\times G_n^1\times (G_n^2)_\Delta$, that is, the left and right
action of $G^1_n$, and the conjugation action of $G^2_n$.

In order to prove the theorem, what we must do is consider the
$G^{1,1,2}_{n,n,n}$-equivariant global chromatographic complex of
$\F_D$ as a $H^*(BG_n^1)$-bimodule, and show that it
matches that of an untwisted strand (the diagram denoted $D'$ in
(\ref{reid1})).

 Note that for any $G_n$ sheaf $\F$ on any $G_n$-space $X$, the inclusion of the symmetric group as permutation matrices normalizing $T_n$ gives an action of $S_n$ on $\hc_{T_n}(X;\res^{G_n}_{T_n}\F)$.

\begin{lemma}\label{localize}
  The natural transformation of functors
  \begin{equation*}
\hc_{G^{1,1,2}_{n,n,n}}(G_{2n}; -)\to \hc_{G^{1,1}_{n,n}\times T^2_n}(G_{2n}; \res^{G^{1,1,2}_{n,n,n}}_{G^{1,1}_{n,n}\times T^2_n}-)
\end{equation*}
is the inclusion of the $S_n$-invariants for the permutation action on $T^2_n$.
\end{lemma}
\begin{proof}
  This is the abelianization theorem for equivariant
  cohomology.
\end{proof}

Let $\hat U$ be the Bruhat cell $Bw^{n,n}_{2n}B$ where $w^{n,n}_{2n}$ is the permutation which switches $i$ and $i\pm n$, and let $\hat j$ be its inclusion to $G_{2n}$.  We note that $\hat j_*\uk_{\hat U}$ is $\fF_{\si}$ where $\si$ is the braid given by the $n$-cabling of a single crossing:
\begin{equation*}
  \xy 0;/r.15pc/: 
   (0,30)="t1"; (5,30)="t2"; (25,30)="t3"; (30,30)="t4"; 
(50,30)="t5"; (55,30)="t6"; (75,30)="t7"; (80,30)="t8";
   (0,0)="b1"; (5,0)="b2"; (25,0)="b3"; (30,0)="b4";
   (50,0)="b5"; (55,0)="b6"; (75,0)="b7"; (80,0)="b8";
   "b1"; "t5" **\dir{-}; "b2"; "t6" **\dir{-}; "b3"; "t7" **\dir{-};
   "b4"; "t8" **\dir{-};  "b5"+(-8,4.8); "b5" **\dir{-}; "b5"+(-20,12); "b5"+(-15,9) **\dir{-}; "t1";"t1"+(23,-13.8) **\dir{-};
 "b6"+(-10.5,6.3); "b6" **\dir{-}; "b6"+(-22.5,13.5); "b6"+(-17.5,10.5) **\dir{-}; "t2";"t2"+(20.5,-12.3) **\dir{-};
 "b7"+(-20.5,12.3); "b7" **\dir{-}; "b7"+(-32.5,19.5); "b7"+(-27.5,16.5) **\dir{-}; "t3";"t3"+(10.5,-6.3) **\dir{-};
"b8"+(-23,13.8); "b8" **\dir{-}; "t4"+(15,-9); "t4"+(20,-12) **\dir{-}; "t4";"t4"+(8,-4.8) **\dir{-};
(22.5,5)*{\cdots};(57.5,5)*{\cdots}; (22.5,25)*{\cdots};(57.5,25)*{\cdots};
(15,-6)*{\underbrace{\hspace{70pt}}_{n\text{ strands}}};(65,-6)*{\underbrace{\hspace{70pt}}_{n\text{ strands}}};
\endxy
\end{equation*}
\begin{lemma}
  The $G_{n,n}^{1,1}\times T_n^2$-equivariant global
  chromatographic complex of $j_*\uk_{U}$ is isomorphic to the
  $T_{n,n}^{1,1}\times T_n^2$-equivariant for $\hat j_*\uk_{\hat U}$,
  with the bimodule structure restricted to $H^*(BG_{n,n}^{1,1})\subset
  H^*(BT_{n,n}^{1,1})$.
\end{lemma}
\begin{proof}
 Let $Q=G_n^1\cap B$ be the upper-triangular matrices in $G_n$, given
  the natural embedding in $G_{n,n}$.  Then
  \begin{equation*}
\ind_{T_{n,n}^{1,1}\times T_n^2}^{G_{n,n}^{1,1}\times T_n^2}j_*\uk_{\hat U}\cong \ind_{Q\times Q\times T_n^2}^{G_{n,n}^{1,1}\times T_n^2}\ind_{T_{n,n}^{1,1}\times T_n^2}^{Q\times Q\times T_n^2}j_*\uk_{\hat U}\cong \res^{G_{n,n,n}^{1,1,2}}_{T_{n,n}^{1,1}\times T_n^2}j_*\uk_{U}
\end{equation*}
The first induction leaves chromatographic complexes unchanged, which $Q$ and $T^1_n$ are homotopy equivalent, and $j_*\uk_{\hat U}$ is smooth on $Q\times Q$-orbits.

For the second, we have a projective map
\begin{equation*}
  \mu:G_n\times_Q\overline{\hat U}\times_Q G_n\to G_{2n}
\end{equation*}
which induces an isomorphism
\begin{equation*}
  G_n\times_Q{\hat U}\times_Q G_n\cong U.
\end{equation*}
  By \cite[Theorem 5]{WWequ}, under taking equivariant
  cohomology, induction of sheaves corresponds to the restriction of
  scalars, and since $G_n/Q$ is projective this result extends to all terms in the chromatographic spectral sequence.
\end{proof}

Of course, by definition, the $T_{n,n}^{1,1}\times T_n^2$-equivariant
chromatographic complex for $\hat j_*\uk_{\hat U}$ is just the complex
of bimodules for the tangle diagram $D_{cab}$ corresponding to
closing the right half of the strands in the braid above.  Applying
the invariance result for labelings all with 1's, this is the same as
the complex corresponding to a full twist of $n$ strands.

Note that if we consider a negative crossing, we will have to include
$n$ times the usual shift for removing a negative stabilization, but
this is easily accounted for in the normalization.

Of course, restricted to symmetric polynomials (that is, $H^*(BG_n)$),
every Soergel bimodule is a number of copies of the regular bimodule,
and every map in the complex for a single crossing splits, so
restricted to $H^*(BG_n)$, the complex attached to a braid labeled all
with 1's is homotopic to a single copy of $H^*(BT_n)$ with the regular
bimodule action and standard $S_n$-action.  By Lemma \ref{localize}, to obtain the
$G^{1,1,2}_{n,n,n}$-equivariant global chromatographic complex we
simply take $S_n$-invariants and thus we obtain a single copy of the
regular bimodule for $H^*(BG_n)$, as desired.
\end{proof}

\bibliography{./gen}
\bibliographystyle{amsalpha}

\end{document}